\documentclass[]{amsart}
\usepackage[utf8]{inputenc}
\usepackage{graphicx}

\DeclareUnicodeCharacter{2212}{-}

\usepackage{mathrsfs}
\usepackage{booktabs}
\usepackage{graphicx}

\usepackage{amsmath, amssymb}
\usepackage{amsthm}
\usepackage{enumerate}
\usepackage{mathtools}
\usepackage{tikz-cd}
\usepackage{comment}
\usepackage{color}
\usepackage{hyperref}
\usepackage{soul}

\newcommand{\CAT}{\operatorname{CAT}}
\newcommand{\GL}{\operatorname{GL}}

\newcommand{\p}{\operatorname{\mathbb{P}}}
\renewcommand{\P}{\operatorname{\mathbb{P}}}

\newcommand{\Z}{\operatorname{\mathbb{Z}}}
\newcommand{\h}{\operatorname{\mathbb{H}^{\infty}}}

\newcommand{\N}{\operatorname{\mathbb{Z}_{\geq 0}}}

\newcommand{\Ind}{\operatorname{Ind}}
\newcommand{\A}{\operatorname{\mathbb{A}}}

\newcommand{\F}{\operatorname{\mathbb{F}}}
\newcommand{\kk}{k}

\newcommand{\id}{\operatorname{id}}

\newcommand{\cent}{\operatorname{Cent}}

\newcommand{\aut}{\operatorname{Aut}}

\newcommand{\Bir}{\operatorname{Bir}}
\newcommand{\Cr}{\operatorname{Cr}}

\newcommand{\SL}{\operatorname{SL}}

\newcommand{\B}{\mathcal{B}}
\DeclareMathOperator{\Bs}{\mathcal{B}}

\def\dashmapsto{\mapstochar\dashrightarrow}

\newcommand{\exc}{\operatorname{Exc}}

\newcommand{\Exc}{\operatorname{Exc}}

\def\dashmapsto{\mapstochar\dashrightarrow}
\newcommand{\dist}{\operatorname{d}}
\renewcommand{\d}{\operatorname{d}}
\newcommand{\CC}{\mathcal{C}}
\newcommand{\Cb}{\mathscr{C}}
\newcommand{\Psaut}{\operatorname{Psaut}}

\renewcommand{\l}{\ell}
\DeclareMathOperator{\Lk}{Lk}
\DeclareMathOperator{\Min}{Min}
\DeclareMathOperator{\Base}{Base}

\usepackage[shortlabels]{enumitem}
\setlist[enumerate]{label=\rm{(\arabic*)}}
\setlist[enumerate,2]{label=\rm({\it\roman*})}
\setlist[itemize]{label=\raisebox{0.25ex}{\tiny$\bullet$}}

\theoremstyle{plain}
\newtheorem{theorem}{Theorem}[section]
\newtheorem{lemma}[theorem]{Lemma}
\newtheorem{proposition}[theorem]{Proposition}
\newtheorem{corollary}[theorem]{Corollary}

\theoremstyle{definition}
\newtheorem{definition}[theorem]{Definition}

\newtheorem{example}[theorem]{Example}

\newtheorem{remark}[theorem]{Remark}
\newtheorem{question}{Question}[section]
\begin{document}

\author{Anne Lonjou and Christian Urech}
\title{Actions of Cremona groups on $\CAT(0)$ cube complexes}
\date{\today}
 \address{Anne Lonjou, Université Paris-Saclay, Département de mathématiques d'Orsay, F-91405 Orsay Cedex, France}
\email{anne.lonjou@u-psud.fr}
\address{Christian Urech, EPFL, SB MATH, Station 8, CH-1015 Lausanne, Switzerland}
\email{christian.urech@epfl.ch}


\thanks{The first author acknowledges support from the Swiss National Science Foundation Grant ``Birational transformations of threefolds'' $200020\mathunderscore178807$, the second author was supported by the Swiss National Science Foundation project $\text{P2BSP}\mathunderscore2175008$.}
\begin{abstract}
	For each $d$ we construct $\CAT(0)$ cube complexes on which Cremona groups of rank $d$ act by isometries. From these actions we deduce new and old group theoretical and dynamical results about Cremona groups. In particular, we study the dynamical behaviour of the irreducible components of exceptional loci. This leads to proofs of regularization theorems, such as the regularization of groups with property FW. We also find new constraints on the degree growth for non-pseudo-regularizable birational transformations, and we show that the centralizer of certain birational transformations is small.
\end{abstract}

\maketitle

%\tableofcontents
\section{Introduction}

To a variety $X$ over a field $\kk$, i.e., an integral and separated scheme of finite type over $\kk$, one can associate its group of birational transformations $\Bir(X)$. If $X$ is rational and of dimension~$d$, $\Bir(X)$ has a particularly rich structure and it is called the \emph{Cremona group of rank $d$}. 
In the last decades Cremona groups have attracted considerable attention. Numerous geometrical, group-theoretical, and dynamical results have been proven for Cremona groups of rank 2 with the help of an action of the group on a hyperbolic space, so that the group is now well understood (see \cite{Cantat_review} for an overview). Unfortunately, there are no interesting actions on hyperbolic spaces available for Cremona groups of rank greater than two. 

Recently, in geometric group theory, actions of groups on $\CAT(0)$ cube complexes have turned out to be important tools to study a large class of groups. These non-positively curved cube complexes naturally appear in numerous instances, yet their combinatorial structure is rigid enough to provide interesting constraints on the structure of the acting group.

The aim of this article is to apply these tools from geometric group theory to Cremona groups. Hence, the main part of this paper consists of constructing $\CAT(0)$ cube complexes on which groups of birational transformations act. Once our cube complexes are constructed, many new and old results about groups of birational transformations will be immediate. Some of our results, especially in the surface case, have already been known before for algebraically closed fields. However, the $\CAT(0)$ cube complexes will give unified and simpler proofs with the additional advantage that they naturally work over arbitrary fields. We also hope that the cube complexes constructed in this article will serve as useful tools for future applications.

In a first part, we construct from a projective regular surface $S$ over a field $\kk$ a cube complex $\Cb(S)$, called \emph{blow-up complex}, on which $\Bir(S)$ acts faithfully by isometries. The vertices of our complex consist of classes of marked surfaces, i.e., projective regular surfaces $T$ together with a birational map $T\dashrightarrow S$, and the edges are given by blow-ups of closed points. This gives a new perspective on $\Bir(S)$. For instance, elements of $\Bir(S)$ that are conjugate to an automorphism of a projective regular surface correspond to elliptic isometries on $\Cb(S)$. More generally, the translation length of an isometry of $\Cb(S)$ induced by $f\in\Bir(S)$ corresponds to the {dynamical number of base points} of $f$ introduced in \cite{Blanc_Deserti}. We will show the following main result:

\begin{theorem}\label{thm:Cbcat0}
	Let $S$ be a projective regular surface over a field $\kk$. The cube complex $\Cb(S)$ is $\CAT(0)$.
\end{theorem}

Unfortunately, there is no direct generalization of this construction for $\Bir(X)$, if the dimension of $X$ is larger than 2 (see Section \ref{subsection_generalization}). Nevertheless, in the second part of this paper, we construct a family of cube complexes $\CC^\l(X)$, for $0\leq\l< ~\dim(X)$, on which $\Bir(X)$ or its subgroups of pseudo-automorphisms act, when $X$ is a variety of any dimension over a field $\kk$. Now, the vertices are classes of marked varieties that are not necessarily complete, as it was the case for the blow-up complex. These constructions are inspired by the paper \cite{cantat-cornulier}, where, for the case $\l=0$ and $k$ algebraically closed, the authors provide a similar construction in the language of commensurating actions. The idea of the construction of our complexes is to add and remove closed subvarieties from $X$. We show:

\begin{theorem}\label{thm:CCcat0}
	Let $X$ be a variety of dimension $d$ over a field $\kk$. For each $0\leq \l\leq d-1$, the cube complex $\CC^\l(X)$ is $\CAT(0)$. 
\end{theorem}

The cube complexes $\CC^\l(X)$ shed light on the dynamics of irreducible components of the exceptional locus of birational transformations, and we obtain, as before, a geometrical interpretation inside our cube complexes of some dynamical properties of groups of birational transformations. For surfaces, this gives a new classification of types of birational transformations that refines, in particular, the previous descriptions of transformations of exponential degree growth under iteration (see Section~\ref{comparison}).

In what follows, let us give an insight into the variety of results that we deduce from the actions of $\Bir(X)$ on our $\CAT(0)$ cube complexes.

\subsection{Regularization results}
An important feature of $\CAT(0)$ cube complexes are the fixed point properties of groups acting on them by isometries. In our case, we will use those properties to deduce regularization theorems. A subgroup $G\subset\Bir(X)$ is called \emph{(pseudo-)regularizable}, if there exists a variety $Y$ and a birational map $f\colon X\dashrightarrow Y$ such that $fGf^{-1}$ is a subgroup of the group of \hbox{(pseudo-)automorphisms} of $Y$. If the variety $Y$ can moreover be chosen to be projective, we call $G$ \emph{projectively regularizable}.

In dimension 2, our constructions will imply immediately various known regularization results and generalize them to arbitrary fields. Moreover, when working over a perfect field $\kk$, we will see that if a subgroup of $\Bir(S)$ is regularizable over the algebraic closure, then it is already regularizable over $k$:

\begin{theorem}\label{prop_reg_field_extension}
	Let $S$ be a geometrically irreducible surface over a perfect field $\kk$ and let $G\subset\Bir(S)$ be a subgroup. Consider the algebraic closure $\bar{\kk}$ of $\kk$ and let $S_{\bar{\kk}}=S\times_\kk\bar{\kk}$. Then the following are equivalent:
	\begin{enumerate}
		\item\label{item_reg_L} There exists a projective regular surface $T_{\bar{\kk}}$ over $\bar{\kk}$ and a $\bar{\kk}$-birational transformation $\varphi\colon T_{\bar{\kk}}\dashrightarrow S_{\bar{\kk}}$ such that $\varphi^{-1}G\varphi\subset\aut(T_{\bar{\kk}})$.
		\item\label{item_reg_k} There exists a projective regular surface $T'$ over $\kk$ and a $\kk$-birational transformation $\varphi\colon T'\dashrightarrow S$ such that $\varphi^{-1}G\varphi\subset\aut(T')$.
	\end{enumerate} 
\end{theorem}

We focus now on certain classes of groups of birational transformations that are always regularizable. A group $G$ has the \emph{property FW} if every action of $G$ on a $\CAT(0)$ cube complex has a fixed point. 
The class of groups with property FW is rather rich and contains in particular all groups with Kazhdan's property (T), such as $\SL_n(\Z)$ for $n\geq 3$ (see for instance \cite{cornulier_wallings}). 

An element $g$ in a group $G$ is called \emph{divisible}, if for every integer $n\geq 0$ there exists an element $f\in G$ such that $f^n=g$. An element $g\in G$ is called \emph{distorted}, if $\lim_{n\to\infty}\frac{|g^n|_{S}}{n}=0$ for some finitely generated subgroup $\Gamma\subset G$ containing $g$, where $|g^n|_{S}$ denotes the word length of $g^n$ in $\Gamma$ with respect to some finite set $S$ of generators of $\Gamma$. We prove in the following theorem that these two classes of birational transformations are regularizable. Distorted elements in $\Bir(\p_\kk^2 )$ for an algebraically closed field $\kk$ of characteristic $0$ have been classified in \cite{blanc2019length} and \cite{cornulier2018distortion} and a description of divisible elements in $\Bir(\p_\kk^2)$ can be found in \cite{liendo2018characterization}.

\begin{theorem}\label{thm_regularization}
	Let $X$ be a variety over a field $\kk$.
	\begin{itemize}
		\item If a subgroup $G\subset\Bir(X)$ has property FW, then $G$ is regularizable and if, moreover, $X$ is a surface, then $G$ is projectively regularizable;
		\item if an element $g\in\Bir(X)$ is divisible or distorted, then $\langle g\rangle$ is regularizable and if, moreover, $X$ is a surface, then $\langle g\rangle$ is projectively regularizable.
	\end{itemize} 
\end{theorem}

The first part of Theorem \ref{thm_regularization} generalizes in particular \cite[Théorème A]{Cantat_groupes_birat}, it implies also \cite[Theorem C]{cantat_xie_padic}, and it answers a question of Serge Cantat and Yves Cornulier (\cite{cantat-cornulier}).

\begin{remark}
	For the case of irreducible and separated schemes of finite type over algebraically closed fields, Yves Cornulier \cite{Cornulier_FW} has obtained simultaneously and with different techniques another proof that groups with property FW are regularizable, using the notion of commensurating actions. 
\end{remark}

\subsection{Centralizers of loxodromic elements} 
Centralizers of elements in the plane Cremona group have been studied in various papers (\cite{Cantat_groupes_birat}, \cite{blanc_cantat_dynamic_degree} \cite{Blanc_Deserti}, \cite{MR3049289}, and the very recent \cite{zhao2019centralizers}), inspired by similar results for diffeomorphism groups. In fact, it has been shown that the centralizers of general elements in the plane Cremona group are virtually cyclic (see \cite{zhao2019centralizers} and references therein). The rigid nature of certain isometries of $\CAT(0)$ cube complexes allows to give new restrictions on centralizers of certain birational transformations of varieties of arbitrary dimension:

\begin{theorem}\label{centralizer}
	Let $X$ be a variety over an algebraically closed field $\kk$ of characteristic $0$. Let $f\in\Bir(X)$ be an element that is not pseudo-regularizable and let $\cent(f)\subset\Bir(X)$ be its centralizer. Then either $f$ permutes the fibers of a rational map $X\dashrightarrow Y$, where $0<\dim(Y)<\dim(X)$, or $\cent(f)$ contains as a finite index subgroup $\langle f\rangle\times H$, where $H\subset\Bir(X)$ is a torsion subgroup. 
\end{theorem}

As mentioned above, in dimension $2$ the torsion group $H$ in Theorem~\ref{centralizer} is always finite, so one could ask whether this is true in arbitrary dimension. {It also seems to be an interesting question whether one could loosen the assumption on the pseudo-regularizability of $f$. }

\subsection{Degree growth} To a birational transformation $f\in\Bir(\p_\kk^2)$ we can associate its degree $\deg(f)$. It is an interesting question how the degrees grow under iteration - a question that is well understood for surfaces (see Theorem~\ref{deggrowth} below), but only little is known in higher dimensions (see \cite{MR3818624} and \cite{cantat2018degrees}). The following theorem shows that the action of $\Bir(\p_\kk^d)$ on our $\CAT(0)$ cube complex gives new lower bounds for the degree growth of a large class of elements:

\begin{theorem}\label{thm_degree_growth}
	Let $\kk$ be any field and let $g\in \Bir(\p_\kk^d)$ be an element that is not pseudo-regularizable. Then the asymptotic growth of $\deg(g^n)$ is at least $\frac{1}{d+1}n$.
\end{theorem}

\subsection*{Organization of the paper}
We start with Section~\ref{Preliminaries}, where we recall preliminaries about $\CAT(0)$ cube complexes and birational transformations in order to fix our notation and state the results needed in this paper. In Section~\ref{section_blow-up_cc}, we construct the blow-up complex $\Cb(S)$ for a regular projective surface $S$, and prove that it is a $\CAT(0)$ cube complex. In a second part we show several applications of this construction and prove the main theorems for surfaces. In Section~\ref{section_cc_higher_ranks} we construct the $\CAT(0)$ cube complexes $\CC^\l(X)$ for $\l=0,\dots,\dim(X){-1}$, where $X$ is an arbitrary variety. As for the blow-up complex, we study its geometry and use it to prove the theorems stated above.
We finish the article with a comparison of different isometric actions of $\Bir(S)$ on non-positively curved spaces, where $S$ is a regular projective surface. We also point out some open questions and possible directions for further research.

\subsection*{Acknowledgements} We thank J\'er\'emy Blanc, Michel Brion, Serge Cantat, Paolo Cascini, Yves Cornulier, R\'emi Coulon, {Thomas Delzant}, Bruno Duchesne, Andrea Fanelli, Anthony Genevois, St\'ephane Lamy, Alexandre Martin, and Immanuel Van Santen for many interesting discussions related to this work. We also thank the referee for the careful lecture and many useful remarks. 

\section{Preliminaries}\label{Preliminaries}
The goal of this section is to gather some notions and results needed about $\CAT(0)$ cube complexes and birational geometry. For details we refer to the standard literature, for instance \cite{sageev_lecturenotes} for $\CAT(0)$ cube complexes, and \cite{Liu} for algebraic geometry.

\subsection{$\CAT(0)$ cube complexes}
Recall that a \emph{cube complex} is the union of finite dimensional Euclidean unit cubes glued together by isometries between some faces.
A cube complex with a global bound on the dimension of its cubes is called \emph{finite dimensional} and its dimension is the maximal dimension of its cubes (see for instance Figure \ref{figure_hyperplanes}). If there is no such global bound, the cube complex is called \emph{infinite dimensional}. An \emph{orientation} of the cube complex is a choice of an orientation on each edge such that every pair of two opposite edges of a square have the same orientation.

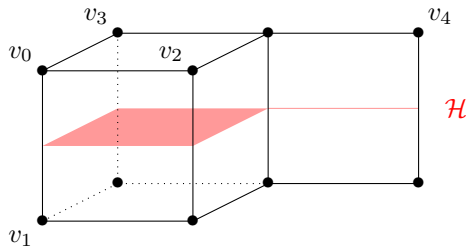
\begin{figure}[h]
	\begin{tikzpicture}
	\fill[color=red!40] (0,1)--(2,1)--(3,1.5)--(1,1.5);
	\draw[color=red!40] (3,1.5)--(5,1.5);
	\draw (0,0) -- (2,0) -- (2,2) -- (0,2)-- (0,0);
	\draw (3,0.5) -- (3,2.5) -- (1,2.5);
	\draw[dotted] (1,2.5)--(1,0.5) -- (3,0.5);
	\draw (2,0) -- (3,0.5); 
	\draw (2,2) -- (3,2.5);
	\draw (0,2) -- (1,2.5);
	\draw[dotted] (0,0) -- (1,0.5);
	\draw (0,0) node {$\bullet$} node[below left] {$v_1$};
	\draw (2,0) node {$\bullet$};
	\draw(2,2) node {$\bullet$} node[above left] {$v_2$};
	\draw (0,2) node {$\bullet$} node[above left] {$v_0$};
	\draw (1,0.5) node {$\bullet$};
	\draw (3,0.5) node {$\bullet$} ;
	\draw(3,2.5) node {$\bullet$};
	\draw (1,2.5) node {$\bullet$} node[above left] {$v_3$};
	\draw (3,0.5) -- (5,0.5)--(5,2.5)--(3,2.5);
	\draw(5,2.5) node {$\bullet$} node[above right] {$v_4$} ;
	\draw (5,0.5) node {$\bullet$};
	\draw(5.5,1.5) node {$\textcolor{red}{\mathcal{H}}$} ;
	\end{tikzpicture}
	\caption{Example of a $\CAT(0)$ cube complex of dimension $3$.\label{figure_hyperplanes}} 
\end{figure}

Let $\mathcal{L}$ be a cube complex. From a vertex $v$ of $\mathcal{L}$, we can construct a simplicial complex $\Lk(v)$ called the \emph{link} of $v$. Its vertices are the vertices of $\mathcal{L}$ that are adjacent to $v$, and a set of vertices $\{v_1,\dots, v_n\}$ in $\Lk(v)$ spans a {$(n-1)$-}simplex in $\Lk(v)$ if $\{v, v_1,\dots, v_n\}$ belongs to the set of vertices of a $n$-cube in $\mathcal{L}$ (for instance in Figure~\ref{figure_hyperplanes}, $\Lk(v_0)$ is a {$2$-}simplex). The link of $v$ is called \emph{flag} if every finite set $\{v_1,\dots, v_n\}$ of vertices of $\Lk(v)$ that are pairwise adjacent, spans a $(n-1)$-simplex in $\Lk(v)$.
A \emph{CAT(0) cube complex} is a simply connected cube complex such that the link of any vertex is flag (see for instance Figure \ref{figure_hyperplanes}). By a theorem of Mikhail Gromov \cite[Section 4.2.C]{Gromov} in the finite dimensional case, and by Ian Leary \cite{leary} in the infinite dimensional case, the Euclidean metric on a cube complex is $\CAT(0)$ if and only if the cube complex is $\CAT(0)$ in the sense of the above definition.

Throughout all the article, we will use the \emph{combinatorial metric}, which is defined on the set of vertices of cube complexes in the following way. Let $\mathcal{V}$ be the set of vertices of a given cube complex $\mathcal{L}$. A (combinatorial) \emph{path of length} $n$ is a sequence of $n+1$ vertices $v_0,\dots, v_n$ such that two successive vertices are adjacent in the cube complex. The distance between two vertices $x,y\in \mathcal{V}$ is the minimal length of all the paths connecting these two vertices. A \emph{geodesic} is a path realizing the distance. Note that in our setting, by connectedness, there exists a geodesic between any pair of vertices, which is not necessarily unique.

The distance between two vertices is better understood in the context of ``hyperplanes''.
Let $\sigma$ be a $n$-dimensional cube in a cube complex. A \emph{midcube} of $\sigma$ is a $(n-1)$-dimensional unit cube passing through the barycenter of $\sigma$ parallely to one of the faces of $\sigma$.
Two edges $e$ and $f$ are {equivalent} if there exists a sequence of edges $\{e_i\}_{0\leq i \leq n}$ such that $e_0=e$, $e_n=f$ and for every $0\leq i \leq n-1$, the two edges $e_i$ and $e_{i+1}$ are the opposite edges of a $2$-cube. We denote the equivalence class of an edge $e$ by $[e]$. Given an equivalence class of edges $[e]$, the \emph{hyperplane} dual to $[e]$ is the collection of all the midcubes that intersect some edges of $[e]$. For instance, in Figure \ref{figure_hyperplanes}, $\mathcal{H}$ is the hyperplane dual to the edge between $v_0$ and $v_1$. By abuse of notation, we will still sometimes denote this hyperplane by $[e]$.

A hyperplane $[e]$ of a $\CAT(0)$ cube complex defines two disjoint half-spaces obtained by taking the two connected components of the complement of $[e]$ inside the cube complex (see \cite[Theorem 4.10]{Sageev-ends_of_groups}). {Once we have chosen an orientation, }we denote them by $[e]^+$ and $[e]^-$, with the convention that $[e]^+$ contains the vertex $v$ if the edge $e=[v,v']$ is oriented from $v$ to $v'$. Consider two vertices $v_1$ and $v_2$ in a $\CAT(0)$ cube complex. A hyperplane $[e]$ \emph{separates} $v_1$ and $v_2$ if they lie in two different half-spaces: either $v_1\in[e]^+$ and $v_2\in[e]^-$, or $v_1\in[e]^-$ and $v_2\in[e]^+$.

\begin{theorem}[{\cite[Theorem 4.13]{Sageev-ends_of_groups}}]\label{theorem:combinatorial_geodesic}
	Let $\mathcal{L}$ be a $\CAT(0)$ cube complex and let $v_1$ and $v_2$ be two vertices. A path joining $v_1$ and $v_2$ is geodesic if and only if the sequence of hyperplanes it crosses has no repetition. In particular, the combinatorial distance between $v_1$ and $v_2$ is equal to the number of hyperplanes of $\mathcal{L}$ that separate $v_1$ and $v_2$.
\end{theorem}

Throughout this paper, we will only consider isometries with respect to the combinatorial metric. For simplicity, we will just call them isometries, in contrast to other authors who call them combinatorial isometries.
An isometry of a $\CAT(0)$ cube complex $\mathcal{L}$ acts on the set of hyperplanes of $\mathcal{L}$. The isometry $f$ has \emph{an inversion along the hyperplane $[e]$} if $f([e]^+)=[e]^-$ (which implies that $f([e])=[e]$ and $f([e]^-)=[e]^+$). We say that $f$ acts \emph{without inversion} if $f$ has no inversion along any hyperplane.

The {\it combinatorial translation length} of an isometry $f$ of the cube complex $\mathcal{L}$ is the number 
\[\l(f)\coloneqq\min\{d(f(x), x)\mid x \in V(\mathcal{L})\},\] where $V(\mathcal{L})$ is the set of vertices of $\mathcal{L}$.
The set of vertices realizing the translation length is called the \emph{minimizing set} of $f$ and it is denoted by
\[\Min(f)\coloneqq\{x\in V(\mathcal{L}) \mid \dist(x,f(x))=\l(f)  \} .\]

An isometry $f$ is called \emph{elliptic} if $\l(f)=0$, and \emph{loxodromic} if it preserves an infinite geodesic path and acts as a {non-trivial} translation along it. Such a geodesic path is called an \emph{axis of $f$}. By \cite[Corollary 5.2]{haglund_isometries_semisimple}, the translation length of $f$ on any of its axis is $\l(f)$, hence the vertices of any axis of $f$ are contained in $\Min(f)$, and every element in $\Min(f)$ lies on an axis of $f$. Moreover, for any $n\in \N$, the isometry $f^n$ is also loxodromic, any axis of $f$ is also an axis of $f^n$ and $\l(f^n)=n\l(f)$. If $f$ is neither elliptic nor loxodromic, it is called \emph{parabolic}. 

These definitions exist also in the context of the  Euclidean metric, instead of the combinatorial one. However, for infinite dimensional cube complexes, these notions do not coincide; there exist isometries that are loxodromic with respect to the combinatorial metric, but parabolic with respect to the $\CAT(0)$-metric (see for instance \cite[Example~3.2]{haglund_isometries_semisimple}). This phenomena occurs when the (combinatorial) axis contains arbitrarily many vertices that are contained in the same cube. But in this paper we always work with the combinatorial metric.

Isometries of $\CAT(0)$ cube complexes have been described by Frédéric Haglund. 
\begin{proposition}[\cite{haglund_isometries_semisimple}]\label{prop_action_semisimple_hag}
	Every isometry $f$ of a $\CAT(0)$ cube complex $\mathcal{L}$ such that $f$ and each of its iterates act without inversion is either loxodromic or elliptic.
\end{proposition}

%The isometry $f$ is called {\it semi-simple} if $\l(f)$ is attained. Proposition~\ref{prop_action_semisimple_hag} shows that an isometry $f$ of a $\CAT(0)$ cube complex such that $f$ and each of its iterates act without inversion is always semi-simple. 
Note that if an isometry $f$ of a $\CAT(0)$ cube complex preserves an orientation, then $f$ and all its iterates act without inversion, hence $f$ is semi-simple.

A useful property of $\CAT(0)$ cube complexes that will be used throughout this paper is given in the following proposition, which has been shown in \cite{gerasimov1998fixed} for the case of finitely generated groups and in \cite[Theorem~11.9]{rollerthesis} for the general case (see also {\cite[Corollary 7.G.4 and Remark 7.F.8]{cornulier_wallings}}).

\begin{proposition}\label{fixedpoint}
	If a group acts isometrically on a $\CAT(0)$ cube complex with a bounded orbit preserving the orientation then the group fixes a vertex.
\end{proposition}

This is a generalization of the well known result that a group acting isometrically on a complete $\CAT(0)$ metric space with {a bounded orbit} has a fixed point.

\subsection{Birational transformations}

Let $X$ and $Y$ be varieties over a field $\kk$. A \emph{rational map} $X\dashrightarrow Y$ is an equivalence class of pairs $(U,g)$, where $U\subset X$ is an open dense set and $g\colon U\to Y$ a morphism. Two pairs $(U, g_1)$ and $(V, g_2)$ are equivalent if $g_1|_{U\cap V}=g_2|_{U\cap V}$. By abuse of notation, we denote a rational map represented by $(U,g)$ just by $g$. 
A rational map represented by $(U, g)$ is \emph{dominant}, if the image of $g$ contains an open dense set of $Y$. If $f\colon Y\dashrightarrow Z$ and $g\colon X\dashrightarrow Y$ are dominant rational maps, we can compose $f$ and $g$ in the obvious way and obtain a dominant rational map $fg\colon X\dashrightarrow Z$. 

A rational map $f\colon X\dashrightarrow Y$ is called a \emph{birational map}, if $f$ is dominant and if there exists a dominant rational map $f^{-1}\colon Y\dashrightarrow X$ such that $ff^{-1}=\id$ and $f^{-1}f=\id$. If there exists a birational map between two varieties $X$ and $Y$, they are said to be \emph{birationally equivalent}. In particular, if $X$ and $Y$ are birationally equivalent, there exist open dense subsets $U\subset X$ and $V \subset Y $ that are isomorphic. For a variety $X $ we denote by $\Bir(X )$ the group of birational transformations from $X $ to itself. If $X $ is a rational variety of dimension $n$, the group $\Bir(X )$ is isomorphic to the Cremona group in $n$ variables $\Cr_n(\kk):=\Bir(\p_\kk^n)$.

The \emph{indeterminacy locus} of a rational map $f\colon X \dashrightarrow Y $ is the closed subset $\Ind(f)\subset X $ consisting of all the points of $X $, where $f$ is not defined, i.e., the points $p\in X $ such that there exists no representative $(U , g)$ of $f$ such that $U $ contains $p$. 
The \emph{exceptional locus} of a birational transformation $f\colon X \dashrightarrow Y $ is the closed subset of points of $X $ where $f$ is not a local isomorphism, i.e., the points that are not contained in any open set $U \subset X $ such that the restriction of $f$ to $U $ induces an isomorphism to the image. 

Let $S $ be a surface over $\kk$ and  $p$ be a closed point on $S$. The \emph{blow-up} of $S$ in $p$ is a variety $Bl_p$ with a morphism $\pi\colon Bl_p\to S $ such that the inverse image of $p$ is a Cartier divisor and such that the following universal property is satisfied: For every morphism $\pi'\colon S '\to S $ such that the inverse image of $p$ is a Cartier divisor, there exists a unique morphism $f\colon S '\to Bl_p$ satisfying $\pi'=\pi f$. Blow-ups always exist and if $S $ is projective, then $Bl_p$ is projective. 

Consider two closed points $p$ and $q$ on the surface $S $. Blowing up first $p$ and then $q$, or first $q$ and then $p$ gives the same surface. Nevertheless we need to define properly how to consider $q$ as a closed point on the surface $Bl_p$.
For this, we need to introduce the so-called bubble space. Let $S $, $T $ and $T' $ be regular surfaces over $\kk$ such that $\pi\colon T \rightarrow S $ and $\pi'\colon T' \rightarrow S $ are surjective birational morphisms. Such surfaces $T $ and $T' $ are said to \emph{dominate} $S $. Two closed points $p\in T $ and $p'\in T' $ are equivalent if $\pi^{-1}\circ \pi'$ is a local isomorphism on a neighborhood of $p'$ and maps $p'$ to $p$. The \emph{bubble space}, denoted by $\B \B(S )$, is the set of all closed points on all regular surfaces over $\kk$ dominating $S $ up to this equivalence relation.

Any birational transformation between projective surfaces can be factorized into blow-ups of closed points in the following strong sense:

\begin{theorem}[ {\cite[\href{https://stacks.math.columbia.edu/tag/0C5Q}{Tag 0C5Q}]{Stack_project}}]\label{factorization}
	Let $S $ and $S '$ be projective regular surfaces over a field $\kk$ and let $f\colon S \dashrightarrow S '$ be a birational transformation. Then there exists a projective surface $T$ {over $\kk$} and two morphisms $\eta\colon T \to S $, $\rho\colon T \to S '$ such that $f=\rho\eta^{-1}$, and we can factorize $\eta\colon T \to S ^n\to\cdots\to S ^1\to S ^0=S $ and $\rho\colon T \to T ^m\to\cdots\to T ^1\to T ^0=S' $, where each morphism is a blow-up in a closed point.
\end{theorem}

If $S $ is a regular surface, then $Bl_p$ is again regular (\cite[SGA6, Expose VII, Proposition 1.8]{SGA6}). So in particular, if $S $ and $S '$ are projective regular surfaces, then all the surfaces $S ^1,\dots, S ^n$, $T , T ^1,\dots, T ^m$ appearing in the factorization given by Theorem~\ref{factorization} are projective and regular.

Consider a finite number of projective and regular surfaces over $\kk$ that are birationally equivalent, Theorem~\ref{factorization} implies in particular that there exists a projective and regular surface over $\kk$ that dominates all of these surfaces.

\begin{remark}\label{remark_unique_factorisation}
	By \cite[Tag 0C5J]{Stack_project} $T $ can be chosen minimal in the sense that for any other projective regular surface $T' $ satisfying Theorem~\ref{factorization}, there exists a surjective morphism $\pi\colon T' \rightarrow T $. This implies in particular that $T $ and $T' $ are respectively obtained from $S $ by blowing up a unique sequence of points in the bubble space $\B\B(S )$.
\end{remark} 
This remark allows us to define the set of base points of $f$.
\begin{definition}
Let $f\colon S \dashrightarrow S' $ be a birational transformation between two projective regular surfaces over a field $\kk$. The base points of $f$ is the set of points of the bubble space that are blown-up by $\eta$  in the minimal resolution of $f$. We denote this set by $\Base(f)$ and its cardinality by $\Bs(f)$.
\end{definition}

\begin{remark}\label{rem:basepoints}
	Let $S$ be a regular projective surface over $k$ and let $\pi\colon T\to S$ be the blow-up of the closed points $p_1,\dots, p_n$ of $S$. Then $\pi$ induces a bijection between $\B\B(T)$ and $\B\B(S)\setminus \{p_1,\dots, p_n\}$. If $f$ is a birational transformation from $S$ to the regular projective surface $S'$ and $T$ a minimal resolution as in Theorem~\ref{factorization}, we therefore obtain a bijection between $\B\B(S)\setminus \Base(f)$ and $\B\B(S')\setminus \Base(f^{-1})$, which we denote, by abuse of notation, also by $f$.
\end{remark}

\begin{definition}
	Let $X$ and $Y$ be varieties over a field $\kk$. A birational map $f\colon X \dashrightarrow Y $ is {\it an isomorphism in codimension $\l$} if the exceptional loci of $f$ and $f^{-1}$ have codimension $>\l$ in $X$ and $Y$ respectively.
\end{definition}

For a variety $X$ of dimension $d$ we denote the group of automorphisms in codimension $\l$ of $X$ by $\Psaut^\l(X )$. In particular, $\Psaut^0(X)=\Bir(X)$ and $\Psaut^{d}(X)=\aut(X)$.

An isomorphism in codimension $1$ is usually called a \emph{pseudo-isomorphism}, like for instance flips and flops between projective varieties. The group of pseudo-automorphism of a variety $X $ will sometimes be denoted by $\Psaut(X )$ instead of $\Psaut^1(X )$.

 We generalize the notion of regularization to groups of pseudo-automorphisms in codimension $\l$.

\begin{definition}
	Let $X $ be a variety over $\kk$. A subgroup $G\subset\Bir(X_k)$ is called \emph{pseudo-regularizable in codimension $\l$}, if there exists a variety $Y $ and a birational transformation $f\colon X \dashrightarrow Y $ such that $fGf^{-1}\subset \Psaut^\l(Y )$.
\end{definition}

In the definition of regularizability in \cite{cantat-cornulier} it is additionally requested that the variety $Y $ is quasi-projective. This looks a priori more restrictive, however, the next lemma shows that the definition from \cite{cantat-cornulier} is equivalent to our definition of regularizability: 

\begin{lemma}\label{benoistlemma}
	Let $X $ be a variety and $G\subset\aut(X )$ a subgroup. Then there exists a quasi-projective variety $Y $ and a birational transformation $f\colon X \dashrightarrow Y $ such that $fGf^{-1}\subset \aut(Y )$.
\end{lemma}

\begin{proof}
	After possibly normalizing we may assume that $X$ is a normal variety.
	By \cite[Theorem~9]{Benoist_normal_varieties}, the variety contains finitely many maximal open quasi-projective subvarieties $U_1,\dots, U_n$. Let $Y$ be the intersection $\bigcap U_i$. In particular, $Y$ is quasi-projective and birationally equivalent to $X$. Since $\aut(X)$ permutes the varieties $U_1,\dots, U_n$, it preserves $Y$. The claim follows.
\end{proof}

\section{The blow-up complex}\label{section_blow-up_cc}

\subsection{Construction}
In her PhD-thesis \cite{lonjouthesis}, the first author constructed a $\CAT(0)$ cube complex $\Cb$ on which the plane Cremona group over an algebraically closed field acts faithfully by isometries. In this section we recall this construction and generalize it to arbitrary projective regular surfaces $S$ over any base field $\kk$, i.e., we construct a $\CAT(0)$ cube complex $\Cb(S)$ on which $\Bir(S)$ acts by isometries.

Given a projective regular surface $S $ over a field $\kk$, we consider pairs $(S' ,\varphi)$, where $S' $ is a projective regular surface over $\kk$ and $\varphi\colon S' \dashrightarrow S $ is a birational transformation. Such pairs are called \emph{marked surfaces}.
The vertices of the complex $\Cb(S )$ are equivalence classes of marked surfaces, where $(T ,\varphi)\sim (T ',\varphi')$ if the map $\varphi'^{-1}\varphi\colon T \to T '$ is an isomorphism. 

A set  $v_1,\dots, v_{2^n}$ of $2^n$ different vertices forms a $n$-cube, if there exist representatives $(S_j,\varphi_j)$ of $v_j$, an $r$, $1\leq r\leq 2^n$ and $n$ different closed points $p_1,\dots, p_n$ on $S_{r}$ such that for all $1\leq i\leq n$ and any blow-up $\pi\colon S'\to S_{r}$ of a subset of $i$ points of $\{p_1,\dots, p_n\}$ there is a $1\leq j\leq 2^n$ such that $(S',\varphi_r\pi)$ is equivalent to $(S_{j},\varphi_j)$. 
For instance, two different vertices $(T ,\varphi)$ and $(T' ,\varphi')$ are linked by an edge if $\varphi^{-1}\varphi'$ is the blow-up or the inverse of the blow-up of a single closed point. 

Recall that to a regular and projective surface $S$ over a field $\kk$, we can associate its Néron-Severi group, which is a finitely generated abelian group. Its rank is called the \emph{Picard rank} and denoted by $\rho(S )$. If we blow up a closed point the Picard rank increases by one. This allows us to define an orientation on $\Cb(S)$. Consider an edge with vertices $(T ,\varphi)$ and $(T' ,\varphi')$. We say that the edge is oriented from $(T, \varphi)$ towards $(T', \varphi')$ if $\rho(T)=\rho(T')+1$.

\begin{lemma}\label{simplyconnected}
Let $S $ be a projective regular surface over a field $\kk$. The cube complex $\Cb(S )$ is simply connected. 
\end{lemma}

\begin{proof}
	We first need to prove that $\Cb(S )$ is connected. Consider two vertices represented by $(T ,\varphi)$ and $(T ',\varphi')$. The map $\varphi'^{-1}\circ \varphi$ is a birational map between projective regular surfaces, so by Theorem~ \ref{factorization} $\varphi'^{-1}\circ \varphi$ factors into a sequence of blow-ups followed by a sequence of contractions. This factorization yields a path from $(T ,\varphi)$ to $(T' ,\varphi')$.

	Let now $\gamma$ be a loop in $\Cb(S )$. We can deform $\gamma$ by a homotopy in such a way that it is contained in the 1-skeleton of the cube complex. Since $\gamma$ is compact, it passes only through finitely many vertices. Let us denote them by $v_1,\dots, v_n$ and choose representatives $(S_{1},\varphi_1),\dots, (S_{n}, \varphi_n)$ of these vertices. 
	We define $\rho_{\text{min}}(\gamma)\coloneqq\min_{1\leq i\leq n}(\rho(S_{i}))$, where $\rho(S_i)$ is the Picard rank of $S_i$. Let $(W,\psi)$ be a vertex that dominates  $(S_{1},\varphi_1),\dots, (S_{n}, \varphi_n)$, i.e., $\varphi_i^{-1}\psi$ is a surjective birational morphism for $1\leq i\leq n$.

 If there exists an $i$ modulo $n$ such that $v_i= v_{i+2}$ then we replace the subpath passing through $v_i,v_{i+1},v_{i+2}$ by the constant path $v_i$ and $n$ decreases by $2$.
 Otherwise, let $1\leq i_0\leq n$ be such that $\rho(S_{i_0})=\rho_{\text{min}}(\gamma)$. This implies that $\rho(S_{i_0-1})=\rho(S_{i_0+1})=\rho(S_{i_0})+1$ and hence that there exist two distinct closed points $p$ and $q$ on $S_{i_0}$ such that $S_{i_0+1}$ is obtained by blowing up $p$ and $S_{i_0-1}$ by blowing up $q$. Let $\pi\colon S'_{i_0}\to S_{i_0}$ be the blow-up of $p$ and $q$ and let $v_{i_0}'$ be the vertex given by $(S_{i_0}', \varphi_{i_0}\pi)$. Since $v_{i_0}'$, $v_{i_0+1}$, $v_{i_0}$ and $v_{i_0-1}$ form a square, we can deform $\gamma$ by a homotopy such that it passes through $v_{i_0}'$ instead of $v_{i_0}$. 
% If $v_{i_0}'$ equals $v_{i_0+2}$ or $v_{i_0-2}$ the subpath $[v_{i_0-2}, v_{i_0}']$ or $[v_{i_0}',v_{i_0+2}]$ of $\gamma'$ is homotopic to the constant subpath at $v_{i_0}'$, so we can retract it and renumerate the vertices from $1$ to $n-2$. 
We have now either that $\rho_{\text{min}}(\gamma)$ increases or that the number of $i$ such that $\rho(S_{i})=\rho_{\text{min}}(\gamma)$ decreases. We observe that the {vertex $(W,\psi)$} dominates also  ${(S_{i_0}', \varphi_{i_0}\pi)}$. We now repeat these steps until we obtain the constant path.
This process terminates because $\rho_{\text{min}}(\gamma)$ is bounded above by $\rho(W)$.
\end{proof}

\begin{lemma}\label{flag}
Let $S $ be a projective regular surface over a field $\kk$. The links of the vertices of the cube complex $\Cb(S )$ are flag.
\end{lemma}

\begin{proof}
	Let $v$ be a vertex of $\Cb(S )$. The vertices of the link $\Lk(v)$ can be identified with the set of edges $\{v, w\}$ in $\Cb(S )$. Let $\{v, w_1\},\dots, \{v, w_n\}$ be a set of pairwise adjacent vertices in $\Lk(v)$, i.e., the vertices $v, w_i, w_j$ belong to a square in $\Cb(S )$. We want to show that $\{v, w_1\},\dots, \{v, w_n\}$ span a simplex in $\Lk(v)$, i.e., that the vertices $v, w_1,\dots, w_n$ belong to a cube.

	The Picard rank of the $w_i$ differs from the Picard rank of $v$ by exactly one. We denote by $m\in\{0,\dots, n\}$ the index such that, up to relabeling the vertices, $\rho(w_i)=\rho(v)+1$ for $ i<m$, and $\rho(w_j)=\rho(v)-1$ for $m\leq j$.
	
	Let $T$ be the surface corresponding to $v$ and for $1\leq i \leq n$ let $S_{i}$ be the surface corresponding to the vertex $w_i$. For $m<j\leq n$, the edges $\{v, w_j\}$ correspond to the contraction of some curves $E_j$ on $T$. Since the $\{v,w_j\}$ are pairwise adjacent in $\Lk(v)$, it follows that the $E_j$ are pairwise disjoint. Similarly, for $1\leq i\leq m$, the edges $\{v,w_i\}$ correspond to the blow-up of different closed points $p_i$ on $T $. Moreover, using again that the $\{v, w_i\}$ are pairwise adjacent in $\Lk(v)$, none of the points $p_i$ for $1\leq i\leq m$ can lie on one of the curves $E_j$ for $m<j\leq n$. This implies that there exists a cube containing $\{v, w_1,\dots, w_n\}$.
\end{proof}

\begin{proof}[Proof of Theorem~\ref{thm:Cbcat0}]
	This follows from Lemma \ref{simplyconnected} and Lemma \ref{flag}.
\end{proof}

\subsection{Hyperplanes}
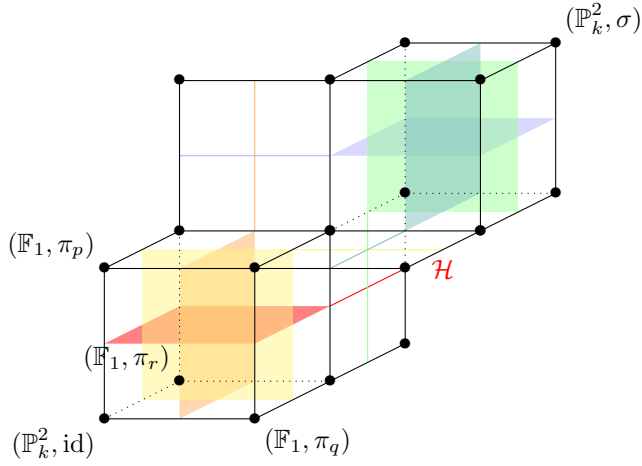
\begin{figure}
	\begin{tikzpicture}
	\fill[opacity=0.5, color=red] (0,1)--(2,1)--(3,1.5)--(1,1.5);
	\draw[color=red] (3,1.5)--(4,2);
	\fill[opacity=0.5, color=orange!60] (1,2)--(2,2.5)--(2,0.5)--(1,0);
	\draw[color=orange!60] (2,2.5)--(2,4.5);
	\fill[opacity=0.5, color=yellow!60] (0.5,2.25)--(2.5,2.25)--(2.5,0.25)--(0.5,0.25);
	\draw[color=yellow!60] (2.5,2.25)--(4.5,2.25);
	\fill[opacity=0.5, color=blue!30] (3,3.5)--(5,3.5)--(6,4)--(4,4);
	\draw[color=blue!30] (3,3.5)--(1,3.5);
	\fill[opacity=0.5, color=green!40] (3.5,4.75)--(5.5,4.75)--(5.5,2.75)--(3.5,2.75);
	\draw[color=green!40] (3.5,2.75)--(3.5,0.75);	
	\fill[opacity=0.5, color=teal!50] (4,4.5)--(5,5)--(5,3)--(4,2.5);
	\draw[color=teal!50] (4,2.5)--(3,2);	
	\draw (0,0) -- (2,0) -- (2,2) -- (0,2)-- (0,0);
	\draw (3,0.5) -- (3,2.5) -- (1,2.5);
	\draw[dotted] (1,2.5)--(1,0.5) -- (3,0.5);
	\draw (2,0) -- (3,0.5); 
	\draw (2,2) -- (3,2.5);
	\draw (0,2) -- (1,2.5);
	\draw[dotted] (0,0) -- (1,0.5);
	\draw (0,0) node {$\bullet$} node[below left] {$(\P_\kk^2,\id)$};
	\draw (2,0) node {$\bullet$} node[below right] {$(\F_1,\pi_q)$} ;
	\draw(2,2) node {$\bullet$} ;
	\draw (0,2) node {$\bullet$} node[above left] {$(\F_1,\pi_p)$};
	\draw (1,0.5) node {$\bullet$} node[ above left] {$(\F_1,\pi_r)$};
	\draw (3,0.5) node {$\bullet$} ;
	\draw(3,2.5) node {$\bullet$} ;
	\draw (1,2.5) node {$\bullet$} ;
	\draw (3,0.5) -- (4,1)--(4,2);
	\draw[dotted] (4,2)--(4,3);
	\draw(5,2.5) node {$\bullet$} node[above right] {} ;
	\draw (4,1) node {$\bullet$} ;
	\draw(4.5,2) node {$\textcolor{red}{\mathcal{H}}$} ;
	\draw (3,2.5) -- (5,2.5) -- (5,4.5) -- (3,4.5)-- (3,2.5);
	\draw (6,3) -- (6,5) -- (4,5);
	\draw[dotted] (4,5)--(4,3) -- (6,3);
	\draw (5,2.5) -- (6,3); 
	\draw (5,4.5) -- (6,5);
	\draw (3,4.5) -- (4,5);
	\draw[dotted] (3,2.5) -- (4,3);
	\draw (3,2.5) node {$\bullet$} node[below left] {};
	\draw (5,2.5) node {$\bullet$} ;
	\draw(5,4.5) node {$\bullet$} ;
	\draw (3,4.5) node {$\bullet$} node[above left] {};
	\draw (4,3) node {$\bullet$} ;
	\draw (6,3) node {$\bullet$} ;
	\draw(6,5) node {$\bullet$} node[above right] {$(\P_\kk^2,\sigma)$} ;
	\draw (4,5) node {$\bullet$} ;
	\draw (1,2.5) -- (1,4.5)-- (3,4.5);
	\draw (2,2) -- (4,2)-- (5,2.5);
	\draw (4,2) node {$\bullet$} ;
	\draw (1,4.5) node {$\bullet$} ;
	\end{tikzpicture}
	\caption{A subcomplex of $\Cb(\P_\kk^2)$ containing the vertices $(\P_\kk^2,\id)$ and $(\P_\kk^2,\sigma)$, where $\sigma$ is the standard quadratic involution.\label{Figure_blowup_complex}} 
\end{figure}
Let $S $ be a projective regular surface over a field $\kk$. Recall that in the complex $\Cb(S )$ an edge is given by the blow-up of a regular point $p$ belonging to a marked surface $(W ,\varphi)$. We denote this edge by $(W ,\varphi,p)$ and by $[(W ,\varphi,p)]$ its equivalence class as well as the hyperplane dual to it.  

\begin{lemma}\label{lemme:equivalence_edges-surface}
	Two edges $(W ,\varphi,p)$ and $(W',\varphi',q)$ correspond to the same hyperplane if and only if $\varphi'^{-1}\varphi$ induces a local isomorphism between a neighborhood of $p$ and a neighborhood of $q$ and $\varphi'^{-1}\varphi(p)=q$.
\end{lemma}

For instance in the example of Figure~\ref{Figure_blowup_complex}, the hyperplane $\mathcal{H}$ corresponds to the hyperplane  $[(\P_\kk^2,\id,p)]=[(\F_1,\pi_q,\pi_q^{-1}(p))]$, {where $\F_1$ denotes the first Hirzebruch surface}. 

\begin{proof}
	Consider the minimal resolution of $\varphi'^{-1}\varphi$. Fixing an order in the blow-ups and the blow-downs gives a path $\gamma$ between the vertices $[(W ,\varphi)]$ and $[(W',\varphi')]$. If $\varphi'^{-1}\varphi$ is a in a neighborhood of $p$, the blow-ups of $p$ in the surfaces corresponding to the different vertices of $\gamma$ give a sequence of equivalent edges. Moreover, because  $\varphi'^{-1}\varphi(p)=q$, the last one corresponds to blow up the point $q$ in $W'$.
	
	Conversely, if $(W ,\varphi,p)$ and $(W',\varphi',q)$ correspond to the same hyperplane, there exists a path $\gamma$ from  $[(W ,\varphi)]$ to $[(W',\varphi')]$ that does not cross the hyperplane $[(W ,\varphi,p)]$. This shows that $\varphi'^{-1}\varphi$ induces a local isomorphism between a neighborhood of $p$ and a neighborhood of $q$.
\end{proof}

\begin{lemma}\label{lemma:comb_geodesic}
	Consider two vertices $(T ,\varphi)$ and $(T' ,\varphi')$ of $\Cb(S )$. The combinatorial distance between these two vertices is equal to:
	\[ \dist \big((T ,\varphi),(T' ,\varphi') \big)= \B(\varphi'^{-1}\varphi)+ \B(\varphi^{-1}\varphi').\] 
 Moreover, every geodesic path joining these two vertices crosses exactly all the hyperplanes corresponding to the base points of $\varphi'^{-1}\varphi$ and to the blow-down of the exceptional divisors corresponding to the base points of $\varphi^{-1}\varphi'$.
\end{lemma}
For instance, in Figure \ref{Figure_blowup_complex}, we see all the possible geodesics joining the vertices $[(\P_\kk^2,\id)]$ and $[(\P_\kk^2,\sigma)]$.

\begin{proof}
	Choosing an order of blowing up the base points and blowing down the contracted curves in the minimal resolution of $\varphi'^{-1}\varphi$ gives us a path joining the vertices $(T ,\varphi)$ and $(T' ,\varphi')$. By Theorem~\ref{theorem:combinatorial_geodesic} and because the resolution is minimal, this path is geodesic. Moreover, any other geodesic segment joining these two vertices has to cross the same hyperplanes.
\end{proof}

{\begin{remark}\label{rmk_pt-base-inverse}
	Note that by comparing the Picard ranks of the surfaces that appear in Theorem \ref{factorization}, one can see that $\B(f)=\B(f^{-1})$. 
	\end{remark}}

\begin{proposition}\label{prop:halfspace}
	Consider a hyperplane $[(W ,\varphi,p)]$ in $\Cb(S )$. The set of vertices $(T ,\varphi_1)$ such that $p$ is a base point of $\varphi_1^{-1}\varphi$ determines the half-space $[(W ,\varphi,p)]^+$. This characterization of half-spaces does not depend on the representative $(W ,\varphi,p)$.
\end{proposition}
For instance, let $f\in\Bir(S )$ and $p\in S $ be a base point of $f$. Then $(S ,\id)\in[(S ,\id,p)]^-$ and $(S ,f^{-1})~\in~[(S ,\id,p)]^+$ (see Figure \ref{Figure_blowup_complex} for an illustration in the case where $f=\sigma$ and $W=S=\P_\kk^2$).

\begin{proof}
	 By Lemma \ref{lemme:equivalence_edges-surface} any representative $(W ',\varphi',p')$ of $[(W ,\varphi,p)]$ has the property that $\varphi^{-1}\varphi'$ is a local isomorphism from a neighborhood of $p'$ into a neighborhood of $p$ with $p=\varphi^{-1}\varphi'(p')$ so $p$ is a base point of $\varphi_1^{-1}\varphi$ if and only if $p'$ is a base point of $\varphi_1^{-1}\varphi'$. This shows that our characterization is independent of the choice of the representative of the class $[(W ,\varphi,p)]$.

	We denote by $E$ the set of vertices $(T ,\varphi_1)$ such that $p$ is a base point of $\varphi_1^{-1}\varphi$ and by $E^C$ its complement. It is enough to prove that any geodesic joining one vertex in $E$ to a vertex in $E^C$ crosses the hyperplane $[(W ,\varphi,p)]$. Let $(T ,\varphi_1)\in E$ and $(T' ,\varphi_2)\in E^C$. By definition, $p$ is not a base point of $\varphi_2^{-1}\varphi$. Assume first that $\varphi_2^{-1}\varphi$ is a local isomorphism in a neighborhood of $p$. Then, because $p$ is a base point of $\varphi_1^{-1}\varphi$, we obtain that $\varphi_2^{-1}\varphi(p)$ is a base point of $\varphi_1^{-1}\varphi_2$. By Lemma \ref{lemma:comb_geodesic}, we conclude that any geodesic joining $(T ,\varphi_1)$ and $(T' ,\varphi_2)$ crosses the hyperplane $[(W ,\varphi,p)]=[(T' ,\varphi_2,\varphi_2^{-1}\varphi(p))]$.
	If $\varphi_2^{-1}\varphi$ is not a local isomorphism in a neighborhood of $p$ this implies that there exists a vertex $(\tilde{T} ,\varphi_3)$ such that $\varphi_2^{-1}\varphi_3$ is a birational morphism and $\varphi_3^{-1}\varphi$ is a local isomorphism between a neighborhood of $p$ and a neighborhood of $\varphi_3^{-1}\varphi(p)$. By the same argument as before, $\varphi_3^{-1}\varphi(p)$ is a base point of $\varphi_1^{-1}\varphi_3$ so it is also a base point of $\varphi_1^{-1}\varphi_2$ and any geodesic joining $(T ,\varphi_1)$ and $(T' ,\varphi_2)$ crosses the hyperplane $[(W ,\varphi,p)]=[(\tilde{T} ,\varphi_3,\varphi_3^{-1}\varphi(p))]$.
\end{proof}

\subsection{Action of the plane Cremona group on the blow-up complex}
Let $S $ be a projective regular surface over a field $\kk$ and let $f\in\Bir(S)$. The action of $f$ on a vertex $v=(T ,\varphi)$ of $\Cb(S )$ is defined as follows: \[f(v)\coloneqq(T , f\varphi).\] This action is well defined and it preserves the combinatorial structure of $\Cb(S )$, i.e., it induces an isometry on $\Cb(S )$. We thus obtain an action of $\Bir(S )$ on $\Cb(S )$, which is faithful. This action induces an action on the set of hyperplanes: for any $f\in\Bir(S )$ and any equivalence class of an edge $[(T ,\varphi,p)]$ we have $f([(T ,\varphi,p)])=[(T ,f\varphi,p)]$. 

For instance, if $\pi_p$ is the blow-up of the closed point $p\in T $, the edge between $(T ,\varphi)$ and $(T ',\varphi\pi_p)$ is mapped by $f$ to the edge between $(T ,f\varphi)$ and $(T' ,f\varphi\pi_p)$. The image of the edge corresponds to the blow-up of the closed point $p$ on the marked surface $(T ,f\varphi)$. This can be seen in the Figure \ref{Figure_blowup_complex}, the hyperplane $\mathcal{H}$ corresponding to $[(\P_\kk^2,\id,p)]$ is sent to the light green one corresponding to $[(\P_\kk^2,\sigma,p)]$.

The action of the Cremona group on the vertices preserves the Picard rank and therefore in particular the orientation of the cube complex. 
%
%
% This implies that the action has no inversion along a hyperplane. Actually, we can see directly even a stronger fact.
%\begin{remark}
%	Consider $f\in\Bir(S )$ and let $[(T ,\varphi,p)]$ be a hyperplane of $\Cb(S )$. It is always possible to find $(T ',\varphi_1)$ such that $p$ is not a base point of $\varphi_1^{-1}\varphi$ and of $\varphi_1^{-1}f^{-1}\varphi$. By Proposition \ref{prop:halfspace}, this implies that $f([(T ,\varphi,p)]^-) \not \subset[(T ,\varphi,p)]^+$.
%	
%	Similarly, it is always possible to find $(\tilde{T} ,\varphi_1)$ such that $p$ is a base point of $\varphi_1^{-1}\varphi$ as well as of $\varphi_1^{-1}f^{-1}\varphi$ implying that $f([(T ,\varphi,p)]^+)\not \subset [(T ,\varphi,p)]^-$.
%\end{remark}
Applying the result of Frédéric Haglund (Proposition \ref{prop_action_semisimple_hag}) we obtain:

\begin{proposition}\label{cor_semisimple}
Every isometry of $\Cb(S )$ that is induced by an element $f$ of $\Bir(S )$ is either combinatorially loxodromic or combinatorially elliptic.
\end{proposition}

In fact, it is possible to show that for all hyperplanes $\mathcal{H}$ and all $f\in\Bir(S)$ the image $f(\mathcal{H}^{\pm})$ is never strictly included in $\mathcal{H}^{\mp}$. 
\subsection{The dynamical number of base points and algebraic stability}Let $S $ be a projective regular surface over a field $\kk$ and $f\in\Bir(S )$. 
The \emph{dynamical number of base points} of $f$ is defined \[
\mu(f)\coloneqq \lim_{n\to\infty}\frac{\B(f^n)}{n}.\]
 The limit always exists since the number of base points is subadditive. For algebraically closed fields of characteristic $0$, this number has been introduced and studied in \cite{Blanc_Deserti}. 
The blow-up complex $\Cb(S )$ gives a geometrical interpretation of the dynamical number of base points. 

\begin{lemma}\label{dynbasepts}
	For all $f\in\Bir(S )$ we have $\l(f)=2\mu(f)$.
\end{lemma}

\begin{proof}By Proposition \ref{cor_semisimple}, for any $f\in\Bir(S)$, for any $x\in \Min(f)$, and for any $n\in\N$ we have $\dist(x,f^n(x))=n\l(f)$.	
	
Let $p=(S , \id)$. Then $\d(p, f^n(p))=2\B(f^n)$ by Lemma \ref{lemma:comb_geodesic}. This implies that for any $n\in \Z_{>0}$, $n\l(f)\leq 2\B(f^n)$. Taking the limit we obtain: $\l(f)\leq 2\mu(f)$. 
On the other hand, let $x\in \Min(f)$ and $K:=\d(p, x)$. Then for any $n\in \N$, $2\B(f^n)=~\d(p, f^n(p))\leq n\l(f)+2K$ and hence $2\mu(f)\leq \l(f)$.
\end{proof}

The dynamical number of base points is invariant by conjugation, since conjugation by a birational transformation $g$ changes the number of base points at most by a constant only depending on $g$. Lemma \ref{dynbasepts} immediately implies the following theorem, which has first been proven by Jérémy Blanc and Julie D\'eserti (\cite{Blanc_Deserti}) for algebraically closed fields of characteristic $0$:

\begin{theorem}\label{blancdesthm}
Let $S $ be a projective regular surface over a field $\kk$ and let $f\in~\Bir(S )$. Then 
	\begin{enumerate}
		\item $\mu(f)$ is an integer;
		\item there exists a smooth projective surface $T $ and a birational map $\varphi\colon T \dashrightarrow S $ such that $\varphi^{-1} f\varphi$ has exactly $\mu(f)$ base points;
		\item in particular, $\mu(f)=0$ if and only if $f$ is conjugate to an automorphism of a regular projective surface.
	\end{enumerate}
\end{theorem}

\begin{proof}
We choose a vertex $(T , \varphi)$ belonging to $\Min(f)$. By Lemma \ref{dynbasepts}, we have $\l(f)=2\mu(f)$. By Lemma \ref{lemma:comb_geodesic}, we obtain that $\l(f)$ is an even integer: $\l(f)=\B(\varphi^{-1}f\varphi)+\B(\varphi^{-1}f^{-1}\varphi)=2\B(\varphi^{-1}f\varphi)$, since $\varphi^{-1} f\varphi$ has the same number of base points as $\varphi^{-1} f^{-1}\varphi$ (see Remark~\ref{rmk_pt-base-inverse}). For the same reason, they have both exactly $\mu(f)$ base points. The translation length is $0$ if and only if $f$ induces an elliptic isometry on $\CC(S)$, which is equivalent to $[T,\varphi]$ fixing a vertex and hence to $f$ being conjugate to an automorphism of $T$.
\end{proof}

Let $S$ be a projective regular surface over a field $\kk$. A birational transformation $f\in \Bir(S )$ is {\it algebraically stable} if for any $n\geq 1$, the base points $\Base(f^{n-1})$ are not contained in $\Base(f^{-1})$ and: \[ \Base(f^n)=\Base(f)\sqcup f^{-1}(\Base(f^{n-1})).\]

An equivalent characterization is that $f\in \Bir(S)$ is algebraically stable if and only if the induced map $f_*$ on the Neron Severi group of $S$ satisfies $(f_*)^n=(f^n)_*$ (see for instance \cite{diller2001dynamics}).
  Jeffrey Diller and Charles Favre showed in \cite{diller2001dynamics} that every birational transformation of a projective surface admits an algebraically stable model.
 The following proposition gives a natural geometrical proof of the theorem of Jeffrey Diller and Charles Favre. 

\begin{proposition}\label{lemma_persistent_axis}
	Let $S$ be a projective regular surface over a field $\kk$ and let $f\in~\Bir(S )$. If $(S' ,\varphi)$ belongs to $\Min(f)$ then $\varphi^{-1}f\varphi$ is algebraically stable. In particular, every birational transformation is conjugate to an algebraically stable transformation. 
\end{proposition}

\begin{proof}{Let $(S ',\varphi)\in \Min(f)$.}
{If $f$ induces an elliptic isometry on $\Cb(S)$ then $(S ',\varphi)$ is fixed by $f$ and $\varphi^{-1} f\varphi\in\aut(S' )$ has no base points. In particular, $\varphi^{-1} f\varphi$ is algebraically stable.}
	
	Now assume that $f$ induces a loxodromic isometry. This means that for any $n\in \N$, $(S ',f^{n}\varphi)$ is on the {same} axis of $f$ and \[\dist((S ',\varphi),(S ',f^{n}\varphi))=n\dist((S ',\varphi),(S ',f\varphi)).\] This implies by Lemma \ref{lemma:comb_geodesic} that $\Bs(\varphi^{-1} f^n\varphi)=n\Bs(\varphi^{-1} f\varphi)$. Moreover, for any birational transformation $g$ of a given variety, it is always true that $$\Base(g^n)\subset\Base(g)\sqcup g^{-1}(\Base(g^{n-1})\setminus \Base(g^{-1})\cap\Base(g^{n-1})).$$ Applying this to $\varphi^{-1} f\varphi$ we obtain that $\varphi^{-1} f\varphi$ is algebraically stable by considering the cardinality of the base points of $\varphi^{-1} f^n\varphi$.
\end{proof}

\subsection{Regularization theorems} First we will establish a criterion for regularizability. Note that in the surface case, if a group of birational transformations is conjugate to a subgroup of $\aut(S)$, we can always assume $S$ to be regular by taking a resolution of singularities.

\begin{proposition}\label{thm_reg_boundedorbit}
	Let $S $ be a projective regular surface over a field $\kk$. A subgroup $G\subset\Bir(S )$ is projectively regularizable if and only if there exists a constant $K$ such that $\B(f)\leq K$ for all $f\in G$.
\end{proposition}

\begin{proof}
	Assume there exists a $K$ such that $\B(f)\leq K$ for all $f\in G$. Consider the vertex $p=[(S , \id)]$ in $\Cb(S)$. For every $f\in G$ we have that $d(p, f(p))\leq 2K$, in particular the orbit of $p$ under $G$ is bounded. By Proposition~\ref{fixedpoint}, $G$ fixes a vertex $(T ,\varphi)$ and is therefore conjugate by $\varphi$ to a subgroup of $\aut(T )$. On the other hand, if $G$ is regularizable, then it fixes a vertex and so its orbits are bounded, in particular the orbit of the vertex $p=[(S , \id)]$. This implies that there exists a constant $K$ such that $\B(f)\leq K$ for all $f\in G$.
\end{proof}

The number of base points of an element $g\in\Bir(\p_\kk^2 )$ is bounded by a constant depending only on the degree of $g$. Hence, as a particular instance of Proposition~\ref{thm_reg_boundedorbit}, we obtain the following well-known result, which is usually shown using results on the Zariski-topology on $\Bir(\p_\kk^2)$, Weil's regularization theorem, and Sumihiro's results on equivariant completions:

\begin{corollary}\label{weilcor}
	Let $\kk$ be a field and $G\subset\Bir(\P_\kk^2 )$ be a group such that $\deg(g)\leq K$ for all $g\in G$, then $G$ is projectively regularizable.
\end{corollary}

We have now also all the necessary tools in order to prove Theorem~\ref{prop_reg_field_extension}:

\begin{proof}[Proof of Theorem~\ref{prop_reg_field_extension}]
	After taking a completion and a resolution of singularities we may assume that $S$ is projective and regular. Since $\kk$ is perfect and since $S$ is geometrically irreducible, the scheme $T'_{\bar{\kk}}$ is a regular {projective} variety over $\bar{\kk}$ (see \cite[p.~90, Remark~2.9, Example~2.10]{Liu}), and $\aut(T')\subset\aut(T'_{\bar{\kk}})$, in particular, if $\varphi G\varphi^{-1}\subset \aut(T')$, then $\varphi G\varphi^{-1}\subset \aut(T'_{\bar{\kk}})$. So \ref{item_reg_k} implies \ref{item_reg_L}.
	
	Now assume that \ref{item_reg_L} is satisfied. We denote by $\Bs_{\kk}(f)$ and $\Bs_{\bar\kk}(f)$ the number of base points of $f$ considered as a birational transformation defined over $\kk$ and $\bar{\kk}$ respectively. By Proposition \ref{thm_reg_boundedorbit}, $\B_{\bar\kk}(f)\leq K$ for all $f\in G$ for some global constant $K$. But for any $f\in G$, we have $\B_{\kk}(f)\leq \B_{\bar\kk}(f)$. Hence, $\B_{\kk}(f)\leq K$ for all $f\in G$, and by Proposition \ref{thm_reg_boundedorbit} we obtain \ref{item_reg_k}.
\end{proof}

\begin{proposition}\label{lemme_FW_surface}
 Let $S $ be a surface over a field $\kk$. Let $G\subset\Bir(S )$ be a subgroup with property FW then it is projectively regularizable.
\end{proposition}

\begin{proof}
	After taking the completion and resolution of singularities, we may assume $S $ to be projective and regular.
	Since $G$ has property FW, its action on $\Cb(S )$ has a fixed point and therefore it fixes also a vertex $[(T ,\varphi)]$. This implies the claim. 
\end{proof}

In \cite[Theorem B]{cantat-cornulier} the authors prove a stronger version of Proposition~\ref{lemme_FW_surface} in the case of algebraically closed fields. Namely, they show that a subgroup $G\subset~\Cr_2(\kk)$ is conjugate to a subgroup of the automorphism group of a minimal surface. However, their methods are more involved and they use some strong results from Vladimir Danilov and Marat Gizatullin, whereas Proposition \ref{lemme_FW_surface} gives a short and more conceptual proof.

\begin{proposition}Let $S $ be a surface over a field $\kk$. Let $G\subset \Bir(S )$ be a subgroup and $H\subset G$ a subgroup of finite index. Then $G$ is regularizable if and only if $H$ is.
\end{proposition}

\begin{proof}
	After completion and resolution of singularities, we may assume $S $ to be projective and regular. Assume that $H$ is regularizable. This implies that there exists a vertex $(T ,\varphi)$ in $\Cb(S )$ which is a fixed point for the action of $H$ on this cube complex. Let $n=[G:H]$ and let $f_1,\dots, f_n$ be representatives of the cosets of $H$ in $G$ and define $K\coloneqq\max\{\B(\varphi^{-1}f_1\varphi),\dots,\B(\varphi^{-1}f_n\varphi)\}$. Note that $K$ does not depend on the choice of representatives. Indeed, if $g=f\mod H$ then $\varphi^{-1}g\varphi=\varphi^{-1}fh\varphi$ for some $h\in H$ and hence $\B({\varphi^{-1}}g{\varphi})=\B({\varphi^{-1}}f{\varphi})$. This implies that $\B(\varphi^{-1}g\varphi)\leq K$ for all $g\in G$. Hence, Proposition~\ref{thm_reg_boundedorbit} implies that $G$ is projectively regularizable. The other direction of the implication is trivial.
\end{proof}

\subsection{About the generalization to higher dimensional varieties}\label{subsection_generalization}

One of the main drawbacks of our cube complexes $\CC^\l(X )$ for higher dimensional $X$, which we construct in the next section, is that its vertices are represented by marked varieties, rather than marked \emph{projective} varieties, as it is the case with the {blow-up} complex $\Cb(X)$ where $X$ is a surface. One could be tempted to try to construct a $\CAT(0)$ cube complex whose vertices are marked projective varieties of dimension~$3$. However, this is not possible. The group of monomial birational transformations of $\p_\kk^3$ is isomorphic to $\GL_3(\Z)$ and as such has the property FW, i.e., every isometric action of $\GL_3(\Z)$ on a $\CAT(0)$ cube complex has a fixed point. However, it is known that the group of monomial transformations on $\p_\kk^3$ is not conjugate to any subgroup of the automorphism group of a projective variety. This can be seen, for instance, by considering the degree sequence of the monomial transformation $(x,y,z)\dashmapsto (yx^{-1},zx^{-1},x)$, which does not satisfy any linear recurrence and is therefore not conjugate to an automorphism of a regular projective threefold (see \cite{MR2358970} for details). 

\section{Cube complexes for varieties of arbitrary dimension}\label{section_cc_higher_ranks}
\subsection{Construction}
Let $X $ be a variety over a field $\kk$ of dimension $d$. The goal is to construct for each $0\leq \l\leq d-1$ a $\CAT(0)$ cube complex $\CC^\l(X )$ on which the group $\Psaut^\l(X )$ of automorphisms in codimension $\l$ of $X $ acts faithfully by isometries. 

We define the vertices of $\mathcal{C}^\l(X )$ to be equivalence classes of marked pairs $(A ,\varphi)$, where $A$ is a  variety over $\kk$ and $\varphi\colon A \dashrightarrow X$ is an isomorphism
in codimension $\l$. Two marked pairs $(A ,\varphi)$ and $(B ,\psi)$ are equivalent if $\varphi^{-1}\psi\colon B \to A $ is an isomorphism in codimension $\l+1$. 

We place an edge between two vertices $v_1$ and $v_2$ (with orientation from $v_1$ to $v_2$), if $v_1$ can be represented by a couple $(A ,\varphi)$ such that there exists an irreducible subvariety $D \subset A $ of codimension $\l+1$ and $v_2$ is represented by $(A \setminus D , \varphi|_{A \setminus D })$. We say that the vertex $v_2$ is dominated by $v_1$ if there exists a sequence of edges connecting $v_1$ to $v_2$ that are all oriented from $v_1$ to $v_2$.

\begin{example}
	Let $X $ be a normal variety and consider $\CC^0(X )$. Let $v_1=[(A ,\varphi)]$ and $v_2=[(B , \psi)]$ be two vertices  and $\psi^{-1}\varphi\colon A \to B$ is a blow-up of an irreducible subvariety $D \subset B $ of codimension $>1$. Then $v_1$ and $v_2$ are connected by an edge with orientation from $v_1$ to $v_2$.
\end{example}

More generally, we define a $n$-cube between $2^n$ distinct vertices $v_1,\dots, v_{2^n}$ if there exist representatives $(A_{i}, \varphi_i)$ of the vertices $v_i$ with the following properties:
\begin{itemize}
	\item There exist an $i$, $1\leq i\leq 2^n$, and $n$ distinct irreducible subvarieties $V_{1},\dots, V_{n}\subset A_{i}$ of codimension $\l+1$.
	\item For each $1\leq j\leq 2^n$ different of $i$, there exists $1\leq m\leq n$ such that $A_{j}=A_{i}\setminus \{V_{i_1}\cup\dots \cup V_{i_m}\}$ for some $1\leq i_1,\dots, i_m\leq n$.
	\item The $\varphi_j$ equals the restriction of $\varphi_i$ to $A_{j}$.
\end{itemize}

With these definitions, for any $0\leq \l<\dim(X )$, the $\CC^\l(X )$ are infinite dimensional cube complexes. Our next goal is to show that they are $\CAT(0)$. We start with an easy remark.

\begin{remark}\label{intersection}
	Let $X $ be a variety and let $[(A_{1},\varphi_1)],\dots, [(A_{n},\varphi_n)]$ be a finite set of vertices in $\CC^\l(X )$. Let $U_{i}\subset A_{i}$ be open dense subsets with complements of codimension $>\l$ such that for $2\leq i\leq n$, $\varphi_i^{-1}\varphi_1$ induces an isomorphism between $U_{1}$ and $U_{i}$. 
	Then the vertex $[(U_{1},\varphi_1|_{U_{1}})]$ is dominated by all the vertices $[(A_{i},\varphi_i)]$. Indeed, $[(U_{i},\varphi_i|_{U_{i}})]$ is dominated by $[(A_{i},\varphi_i)]$,
	and by construction, all the vertices $[(U_{i},\varphi_i|_{U_{i}})]$ coincide.
\end{remark}

\begin{lemma}\label{lemma_v_dominates_cube}
	Let $v, v_1,\dots, v_n$ be vertices in $\CC^\l(X)$ such that $v$ dominates $v_1,\dots, v_n$. Then there exists a cube containing $v,v_1,\dots, v_n$ all of whose vertices are dominated by $v$.
\end{lemma}

\begin{proof}
Let us observe that for any edge connecting a vertex $v$ with vertex $v'$, and for any representative $(A,\varphi)$ of $v$ there exists an irreducible subvariety $D \subset A $ of codimension $\l+1$ and $v'$ is represented by $(A \setminus D , \varphi|_{A \setminus D })$. Because $v$ dominates the $v_i$ we can therefore find by induction a representative of the $v_i$ of the form $(A \setminus D^i, \varphi|_{A \setminus D^i })$ where $D^i$ is a finite union of irreducible subvarieties of codimension $\l+1$ of $A$ for each $i$. Thi shows that $v,v_1,\dots,v_n$ form a cube whose vertices are dominated by $v$.
\end{proof}

\begin{lemma}\label{simplyconnected_l}
	Let $X $ be a  variety. The cube complexes $\CC^\l(X )$ are simply connected.
\end{lemma}

\begin{proof}
	Remark~\ref{intersection} implies that $\CC^\l(X )$ is connected.
	Let now $\gamma$ be a loop in $\CC^\l(X )$. Up to homotopy we may assume that $\gamma$ is contained in the $1$-skeleton of $\CC^\l(X )$. Denote by $v_1,\dots, v_n$ the vertices that are contained in $\gamma$ and denote by $v$ the vertex from Remark~\ref{intersection} that is dominated by all the $v_i$, i.e., we can write $v_i=[(A_{i},\varphi_i)]$, $v=[(A,\varphi)]$, such that $\varphi_i^{-1}\varphi\colon A\to A_{i}$ is an {open immersion}. We will show that $\gamma$ can be contracted to $v$. 
	
	Define $c(v_i)$ to be the number of irreducible components of pure codimension $\l+1$ of $A_{i}\setminus \varphi_i^{-1}\varphi(A)$. Note that $c(v_i)$ does not depend on the representative of $v_i$ and that $c(v_i)=0$ if and only if $v_i=v$. 
	
	Let now $j$ be such that $c(v_j)$ is maximal among all the $v_j$. Since the vertices $v_{j-1}$ and $v_{j+1}$ (here, we take the indices modulo $n$) are connected to $v_j$ by an edge, we have $c(v_{j-1})=c(v_{j+1})=c(v_j)-1$. If $v_{j-1}=v_{j+1}$, we can contract the part of $\gamma$ from $v_{j-1}$ to $v_{j+1}$ by a homotopy to the vertex $v_{j-1}$. Now, assume that $v_{j-1}\neq v_{j+1}$. Then $v_j$ dominates $v_{j-1}$, $v_{j+1}$ and $v$ so by Lemma \ref{lemma_v_dominates_cube} there exists a cube containing $v_v$, $v_{j-1}$, $v_{j+1}$ and $v$. Moreover, this cube contains a square spanned by $v_{j}$, $v_{j-1}$ and $v_{j+1}$. We denote by $v_j'$ the fourth vertex.
We now deform $\gamma$ by a homotopy so that it passes through $v_j'$ instead of $v_j$. By the definition of our cube, $v_j'$ still dominates $v$ and $c(v_j')<c(v_j)$. So either the maximal value $c(v_j)$ decreases, or the number of vertices $v_i$ such that $c(v_i)$ is maximal decreases and we may continue by induction.
\end{proof}

\begin{lemma}\label{separated}
	Let $X $ be a scheme of finite type over a field $\kk$. If all pairs of points $p,q\in X $ are contained in an open separated subscheme $U\subset X $, then $X $ is separated. 
\end{lemma}

\begin{proof}
	We have to show that the diagonal $\Delta_{X }\subset X \times X $ is closed in $X \times X $. Let $x\in X \times X \setminus \Delta_{X }$ and denote by $p$ and $q\in X $ the image of the first and second projection respectively. Let $U \subset X $ be an open separated subscheme containing $p$ and $q$. Since $U $ is separated, the diagonal $\Delta_{U }=\Delta_{X }\cap (U \times U )$ is closed in $U \times U $. Hence $(U \times U )\setminus \Delta_{U }$ is open in $U \times U $ and hence also in $X \times X $, in particular $(X \times X )\setminus \Delta_{X }$ is open.
\end{proof}

%Let us recall the following glueing lemma for schemes (see. for instance, \cite[Section 2.3, Lemma~3.33]{Liu}).
%\begin{lemma}\label{lemma:glueing}
%Let $S$ be a scheme and consider a family $\{X_i\}_i$ of schemes over $S$. Suppose there is a family $\{U_{ij}\}_j$ of open subschemes of $X_i$ and isomorphisms of $S$-schemes $f_{ij}:U_{ij}\rightarrow U_{ji}$ such that $f_{ii}=\id_{X_i}$, $f_{ij}(U_{ij}\cap U_{ik})=U_{ji}\cap U_{jk}$ and $f_{ik}=f_{jk}f_{ij}$ on $U_{ij}\cap U_{ik}$. Then there exists an $S$-scheme $X$, unique up to isomorphism, with open immersions $g_i:X_i\rightarrow X$ such that $g_i=g_jf_{ij}$, and that $X=\underset{i}{\cup}g_i(X_i)$.
%\end{lemma}

\begin{proof}[Proof of Theorem~\ref{thm:CCcat0}]
	By Lemma~\ref{simplyconnected_l}, it remains to show that the links are flag. Let $v=[(A ,\varphi)]$ be a vertex and let $\{v, v_1\}, \{v, v_2\},\dots, \{v, v_n\}$ be edges departing from $v$ that are pairwise contained in a square. We need to show that $v, v_1,\dots, v_n$ are contained in a cube. 
	
	We choose representatives $(A_{i},\varphi_i)$ for the $v_i$. Up to relabeling we may assume that there is a $m$ such that the edges $\{v, v_1\}, \{v, v_2\},\dots, \{v, v_m\}$ are oriented away from $v$, while the edges $\{v, v_{m+1}\}, \dots, \{v, v_n\}$ are oriented towards $v$. In particular,  for $m+1\leq i\leq n$ the exceptional loci of the maps $\varphi_i^{-1}\varphi\colon A \dashrightarrow A_{i}$ are of codimension $>l+1$.

	We now construct a vertex $w$ that dominates the vertices $v_{m+1},\dots, v_n$, by { glueing} together $A_{m+1},\dots, A_{n}$ with the help of the markings. By assumption, for all $r,s\geq m+1$, the vertices $v, v_r, v_s$ are contained in a square. Denote the fourth vertex in that square by $v_{rs}=[(A_{rs}, \varphi_{rs})]$. After removing finitely many loci of codimension $>\l+1$ from the $A_{r}$ and $A_{s}$ for all $r,s\geq m+1$ we may assume that the maps $\varphi_{rs}^{-1}\varphi_{r}\colon A_{r}\dashrightarrow A_{rs}$ and $\varphi_{rs}^{-1}\varphi_{s}\colon A_{s}\dashrightarrow A_{rs}$ are {open immersions} for all $r,s\geq m+1$. 
	For all $r,s\geq m+1$ denote by $U_{rs}\subset A_{r}$ the maximal open dense subsets such that $\varphi_s^{-1}\varphi_r|_{U_{rs}}\colon U_{rs}\to U_{sr}$ is an isomorphism. 
	After possibly removing finitely many loci of codimension $>\l+1 $ from the $A_{rs}$ we may moreover assume that $A_{rs}=\varphi_{rs}^{-1}\varphi^{}_r(A_{r})\cup\varphi_{rs}^{-1}\varphi_s(A_{s})$ for all $r,s\geq m+1$. 
	In other words, $A_{rs}$ is the variety obtained by glueing together $A_{r}$ and $A_{s}$ along $U_{rs}\subset A_{r}$ and $U_{sr}\subset A_{s}$ through the isomorphism $\varphi_s^{-1}\varphi^{}_r|_{U_{rs}}\colon U_{rs}\to U_{sr}$. 

	Consider now the integral $\kk$-scheme $B $ of finite type given by glueing together the varieties $A_{m+1},\dots, A_{n}$ along the isomorphisms $\varphi_s^{-1}\varphi_r^{}|_{U_{rs}}\colon U_{rs}\to U_{sr}$. The glueing maps satisfy the conditions of  the glueing lemma (see for example \cite[Section 2.3, Lemma~3.33]{Liu}), so we indeed obtain a $\kk$-scheme, which is by construction irreducible and of finite type.
	%"Glueing Lemma", by Hartshorne Exercise 2.12
	 Again by the glueing lemma,  we can naturally identify the varieties $A_{r}$ and $A_{rs}$ as open subschemes of $B $. Now, for any pair of points $p,q\in B$ there exist $r,s\geq m+1$ such that $p$ is contained in $A_{r}$ and $q$ is contained in $A_{s}$. In particular, $p$ and $q$ are contained in $A_{rs}$. Lemma~\ref{separated} implies now that $B$ is separated, so $B$ is indeed a  variety. Note as well that the complements of the $A_{r}$ in $B$ are of codimension $\l+1$. Let $\psi_{B}\colon B \dashrightarrow X $ be the  isomorphism in codimension $\l$ whose restriction to $A_{r}$ is $\varphi_r$. Hence, our vertex $w$ is $[(B, \psi_B)]$.
	We conclude using Lemma \ref{lemma_v_dominates_cube} because $w$ dominates all the vertices $v,v_1,\dots,v_n$.
\end{proof}

\subsection{Some remarks on the geometry of $\CC^\l(X)$}
In this subsection we will show that for normal varieties $\CC^0(X)$ consists only of one single cube (of infinite dimension). If $\dim(X)\geq 3$, this is not the case anymore for $\CC^\l(X)$ if $\l>0$.

\begin{lemma}\label{graphlemma}
	Let $X_1,\dots, X_n , Z$ be normal varieties of dimension $d$ over a field $\kk$ and let $\varphi_i\colon X_{i}\dashrightarrow Z$ be birational maps. Then there exists a complete normal variety $Y $ and birational maps $\psi_i\colon Y \dashrightarrow X_{i}$ such that the irreducible components of the exceptional locus of $\psi_i^{-1}$ have codimension $>1$ for all $i$.
\end{lemma}

\begin{proof}
	By induction, it is enough to show the claim for $n=2$. Let $\overline{X_{1}}$ and $\overline{X_{2}}$ be completions of $X_{1}$ and $X_{2}$ and let $\Gamma\subset \overline{X_{1}}\times \overline{X_{2}}$ be the graph of $\varphi_2^{-1}\varphi_1$. Define ${Y }$ to be the normalization of $\Gamma$ and let ${\psi_i}\colon {Y }\to \overline{X_{i}}$ be the projections. By Zariski's main theorem \cite[p. 150, Corollary 4.3]{Liu}, the irreducible components of the exceptional loci of the $\psi_i^{-1}$ have codimension $>1$ and so are their restrictions to $X_{i}$. Hence, the $\psi_i$ define birational maps from ${Y}$ to $X_{i}$ with the required properties.
\end{proof}

\begin{proposition}\label{onecube}
	Let $X $ be a normal variety. Every finite set of vertices in $\CC^0(X )$ is contained in a single cube.
\end{proposition}

\begin{proof}
	Choose representatives $(A_{i},\varphi_i)$ of the vertices $v_i$. By Lemma~\ref{graphlemma}, there exists a complete normal variety $Y $ and birational maps $\psi_i\colon Y\dashrightarrow A_{i}$ such that the irreducible components of the exceptional locus of the $\psi_i^{-1}$ have codimension $>1$ for all $i$. Let $U_{i}\subset A_{i}$ be the open dense subset, on which $\psi_i^{-1}$ is an isomorphism. Since $A_{i}\setminus U_{i}$ is of codimension $>1$, we have $[(A_{i}, \varphi_i)]=[(U_{i}, \varphi_i|_{U_{i}})]=[(\psi_i^{-1}(U_{i}), \varphi_i\psi_i|_{\psi_i^{-1}(U_{i})})]$.

	 Now, let $A=\psi_1^{-1}(U_{1})\cup\dots\cup\psi_n^{-1}(U_{n})$ be the union and $B=\psi_1^{-1}(U_{1})\cap\dots\cap\psi_n^{-1}(U_{n})$ the intersection in $Y $ and let $\psi\colon A\dashrightarrow X $ be the birational transformation that restricts to $\varphi_j\psi_j$ on every $\psi_j^{-1}(U_{j})$. Let $E_1,\dots, E_r$ be the irreducible components of $A\setminus B$ of pure codimension $1$. This implies that the irreducible components of pure codimension $1$ of $A\setminus \psi_i^{-1}(U_{i})$ are contained in $\{E_1,\dots, E_{r}\}$. After possibly removing some closed subsets of codimension $>1$ from the $\psi_i^{-1}(U_{i})$, we may assume that $\psi_i^{-1}(U_{i})=A\setminus\{E_{i_1}\cup\dots E_{i_m}\}$. It follows, by the definition of the $n$-cubes that the vertices $v_1,\dots, v_n$ are contained in the $r$-cube defined by $(A,\psi)$ and the irreducible subvarieties $E_1,\dots, E_r\subset A$ of codimension~$1$. 
	 \end{proof}
 
As a direct consequence of Proposition~\ref{onecube}, we obtain again that $\CC^0(X )$ is $\CAT(0)$. Indeed, Proposition~\ref{onecube} implies directly that $\CC^0(X)$ is simply connected and flag.

\begin{remark}\label{rmq_plusieurs_cubes}
	Proposition~\ref{onecube} suggests that the geometry of $\CC^0(X )$ is easier to understand than the one of $\CC^\l(X )$ for $\l>0$. Let us remark that for $\l>0$ the cube complex $\CC^\l(X)$ doesn't have the property anymore that every finite set of vertices is contained in one cube. We will illustrate this with the example of the \emph{Atiyah flop}. Indeed, let $V$ be the hypersurface in $\P^4$ given by the equation $\{xy=zw\}$ and denote by $\pi^{}_0\colon Z\to V$ the blow-up of the origin in $V$. The exceptional divisor $E$ of this blow-up is isomorphic to $\P_\kk^1\times \P_\kk^1$ and we can contract any of its two rulings to a line. Let us denote these two contractions by $\pi^{}_{L}\colon Z\to X $ and $\pi^{}_{L'}\colon Z\to X' $ and let $L=\pi^{}_L(E)$, $L'=\pi^{}_{L'}(E)$. Since $f:=\pi^{}_{L'}\pi_L^{-1}\colon X\dashrightarrow X'$ induces an isomorphism between the complement $U$ of $L$ in $X$ and the complement of $L'$ in $X'$, it is a pseudo-isomorphism, called the Atiyah flop. Define the following vertices $v_1=[(X , \pi^{}_0\pi_{L}^{-1})]$, $v_2=[(X' , \pi^{}_0\pi_{L'}^{-1})]$ and $v_3=[(U , \pi^{}_0{\pi_{L}^{-1}}\mid_{ U})]$ in $\CC^1(V)$. We claim now that $v_1, v_2$, and $v_3$ can not be contained in a cube. Let us assume, for a contradiction, that there exists a cube in $\CC^1(V)$ containing $v_1, v_2$, and $v_3$. Since the vertex $v_3$ is connected to both $v_1$ and $v_2$ by an edge, this implies in particular that $v_1, v_2,$ and $v_3$ are contained in a square. More precisely, there exists a vertex $v_0$ that is connected by an edge to $v_1$ and $v_2$ such that the vertices $v_0, v_1, v_2, v_3$ form a square. Hence, there exist representatives of each of the vertices $(Y ,\varphi)\in v_0$, $(Y_{1},\varphi\mid_{ Y_{1}})\in v_1$, $(Y_{2},\varphi\mid_{Y_{2}})\in v_2$ and $(Y_{3},\varphi\mid_{ Y_{3}})\in v_3$, and two irreducible curves $C_1,C_2\subset Y $ such that $Y_{i}=Y\setminus C_i$ for $i\in\{1,2\}$ and $Y_{3}=Y\setminus \{C_1\cup C_2\}$. So there exist birational transformations $\psi\colon X\dashrightarrow Y_{1}$ and $\psi'\colon X'\dashrightarrow Y_{2}$ that are isomorphisms in codimension $2$, i.e., there exist finitely many points $\{p_1,\dots, p_n\}$ of $X $ such that $\psi$ induces a morphism $X\setminus\{p_1,\dots, p_n\}\to Y_{1}\subset Y$ as well as finitely many points $\{q_1,\dots, q_m\}$ such that $\psi'$ induces a morphism $X'\setminus \{q_1,\dots, q_m\}\to Y_{2}\subset Y$. Moreover, $\varphi\psi=\pi^{}_0\pi_{L}^{-1}$ and $\varphi\psi'=\pi^{}_0\pi_{L'}^{-1}$. Consider the variety $Z':=Z\setminus\{\pi_L^{-1}(\{p_1,\dots, p_n\})\cup \pi_{L'}^{-1}(\{q_1,\dots, q_m\}) \}$ and the divisor $E':=E\setminus\{\pi_L^{-1}(\{p_1,\dots, p_n\})\cup \pi_{L'}^{-1}(\{q_1,\dots, q_m\})\}$ on $Z'$. By definition, the two morphisms $\psi\pi^{}_{L}\mid_{{Z'}}$ and $\psi'\pi^{}_{L'}\mid_{{Z'}}$ coincide on $Z'\setminus E'$, but $\psi\pi^{}_{L}\mid_{{Z'}}(E')$ is an open dense set of $C_1$, whereas $\psi'\pi^{}_{L'}\mid_{{Z'}}(E')$ is an open dense set of $C_2$, which contradicts the fact that $Y$ is separated.
\end{remark}

\begin{remark}
	If $S$ is a regular projective surface, then the blow-up complex is a subcomplex of $\CC^{0}(S)$. The vertices of the blow-up complex can be identified with the vertices in $\CC^0(S)$ of the form $[(S',\varphi)]$, where $S'$ is a regular projective surface and $2^n$ vertices of this form define an $n$-cube in $\CC^0(S)$ if they define an $n$-cube in the blow-up complex. 
	 However, this injection of the blow-up complex into $\CC^0(S)$ is not 
	 an isometric embedding. 
\end{remark}

\subsection{Hyperplanes}{Let $X$ be a variety over a field $\kk$.}
Fix a birational transformation $f\in \Bir(X )$.
For any $1\leq \l\leq \dim(X )$ we denote by $\exc^\l(f)$ the set of irreducible components of the exceptional locus of $f$ of codimension $\l$. 

{An edge of $\CC^\l(X )$ is given by removing an irreducible subvariety  $H$ of codimension $\l+1$ of a marked variety $(A,\varphi)$ that represents the vertex of  $\CC^\l(X )$. We denote this edge by $(A,\varphi, H)$ and by $[(A,\varphi, H)]$ the hyperplane dual to its equivalence class.}

\begin{lemma}\label{lemma_same_hyperplane}
Two triples $(A,\varphi, H)$ and $(B,\psi, D)$ correspond to the same hyperplane if and only if $D$ is not contained in the exceptional locus of $\varphi^{-1}\psi$ and $H$ is the strict transform $\varphi^{-1}\psi(D)$ of $D$ under $\varphi^{-1}\psi$.
\end{lemma}
\begin{proof}
	First assume that $(A,\varphi, H)$ and $(B,\psi, D)$ define the same hyperplane. Let $e$ and $f$ be the edges defined by $(A,\varphi, H)$ and $(B,\psi, D)$, respectively. By definition, this means that there exists a sequence of edges $e_0,\dots, e_n$ such that $e_0=e$, $e_n=f$ and for all $i$, the edges $e_i$ and $e_{i+1}$ are opposite edges of a $2$-cube. If $n=1$, then $e$ and $f$ are opposite edges of a $2$-cube. It follows that $D$ is not contained in the exceptional locus of $\varphi^{-1}\psi$ and $H$ is the strict transform $\varphi^{-1}\psi(D)$ of $D$ under $\varphi^{-1}\psi$. For $n>1$ we obtain inductively that $D$ is not contained in the exceptional locus of $\varphi^{-1}\psi$ and $H$ is the strict transform $\varphi^{-1}\psi(D)$ of $D$. 
	
	On the other hand, assume that for two given triples $(A,\varphi, H)$ and $(B,\psi, D)$, the subvariety $D$ is not contained in the exceptional locus of $\varphi^{-1}\psi$ and $H$ is the strict transform $\varphi^{-1}\psi(D)$ of $D$ under $\varphi^{-1}\psi$. Let $U_{1}\subset A$ and $U_{2}\subset B$ be the biggest open dense subsets such that $\varphi^{-1}\psi$ induces an isomorphism between $U_{2}$ and $U_{1}$. {Note that $U_{1}$ intersects $H$ in a dense open subset and $U_{2}$ intersects $D$ in a dense open subset. Therefore, the two hyperplanes $[(U_{1}, \varphi|_{U_{1}},H\cap U_{1})]$ and $[(U_{2},\psi|_{U_{2}},D\cap U_{2})]$ coincide. Moreover the triples $(A,\varphi, H)$ and $(U_{1},\varphi|_{U_{1}},H\cap U_{1})$ correspond to the same hyperplane, and the triplets $(B, \psi, D)$ and $(U_{2}, \psi|_{U_{2}}, D\cap U_{2})$ correspond to the same hyperplane. This finishes the proof.}
\end{proof}

\begin{lemma}\label{distancel}
	Let $v_1=[(A,\varphi)]$ and $v_2=[(B,\psi)]$ be two vertices in $\CC^\l(X )$. Then the following is true:

	 The distance between $v_1$ and $v_2$ equals the sum \[
	{\dist(v_1,v_2)=}\lvert\Exc^{\l+1}(\varphi^{-1}\psi)\rvert+\lvert\Exc^{\l+1}(\psi^{-1}\varphi)\rvert.
	\]
	 Moreover, every geodesic path joining these two vertices crosses exactly all the hyperplanes that correspond to removing the irreducible components of pure codimension $\l+1$ of the exceptional locus of $\psi^{-1}\varphi$ from $A$ and adding the irreducible components of pure codimension $\l+1$ of the exceptional locus of $\varphi^{-1}\psi$.
\end{lemma}

\begin{proof}
Removing the irreducible components of pure codimension $\l+1$ of the exceptional locus of $\psi^{-1}\varphi$ from $A$ and then adding the irreducible components of pure codimension $\l+1$ of the exceptional locus of $\varphi^{-1}\psi$ gives us a path between $v_1$ and $v_2$. By Lemma~\ref{lemma_same_hyperplane} and because the resolution is minimal, Theorem~\ref{theorem:combinatorial_geodesic} implies that this path is geodesic. Moreover, any other geodesic segment joining these two vertices has to cross the same hyperplanes. This also implies that the distance between $v_1$ and $v_2$ equals $\lvert\Exc^{\l+1}(\varphi^{-1}\psi)\rvert+\lvert\Exc^{\l+1}(\psi^{-1}\varphi)\rvert$.
	
\end{proof}

 \begin{proposition}\label{prop:halfspace_Cl}
 	Consider the hyperplane $[(A,\varphi, H)]$. The set of vertices $[(B,\psi)]$ such that $H$ is contained in the exceptional locus of $\psi^{-1}\varphi$ determines the half-space $[(A,\varphi, H)]^-$. This characterization of half-spaces does not depend on the representative $(A,\varphi, H)$.
 \end{proposition}

 \begin{proof}
 	 By Lemma \ref{lemma_same_hyperplane}, any representative $(A',\varphi',H')$ of $[(A,\varphi,H)]$ has the property that $H'$ is not contained in the exceptional locus of $\varphi^{-1}\varphi'$ and $H$ is the strict transform of $H'$ under $\varphi^{-1}\varphi'$.
 	We obtain that if $H'$ belongs to $\exc(\psi^{-1}\varphi')$ then $H$ belongs to $(\varphi^{-1}\varphi')\exc(\psi^{-1}\varphi')$, and so $H$ belongs to $\exc(\psi^{-1}\varphi'(\varphi^{-1}\varphi')^{-1})=\exc(\psi^{-1}\varphi)$. Reversing the roles of $\varphi$ and $\varphi'$ we obtain similarly the converse.

 	We denote by $E$ the set of vertices $[(B,\psi)]$ such that $H$ is contained in the exceptional locus of $\psi^{-1}\varphi$, and by $E^C$ its complement. It is enough to prove that any geodesic joining a vertex in $E$ with a vertex in $E^C$ crosses the hyperplane $[(A,\varphi, H)]$. This is a consequence of Lemma \ref{distancel}.
 \end{proof}

\subsection{Action of groups of pseudo-automorphisms in codimension $\l$}
For any $\l$, there is a faithful action of $\Psaut^\l(X)$ on $\CC^\l(X)$, defined in the following way: an element $f\in\Psaut^\l(X)$ maps a vertex $[(A,\varphi)]$ to the vertex $[(A, f\varphi)]$. It is straightforward to check that this action on the vertices is well-defined, preserves the cubes, and preserves the orientation. Consequently, as a corollary to the result of Frédéric Haglund (Proposition~\ref{prop_action_semisimple_hag}) we obtain the semi-simplicity of the action of $\Psaut^\l(X)$ on $\CC^\l(X)$:

\begin{proposition}\label{prop_semisimple_psaut}
	Every isometry of $\CC^\l(X )$ induced by an element $f$ of $\Psaut^\l(X )$ is either combinatorially elliptic or combinatorially loxodromic.
\end{proposition}

\subsection{Dynamics of exceptional loci}\label{subsection_dynamics_exceptional_loci}

For a transformation $f\in\Psaut^{\l}(X)$, we denote by
\[
\nu^{\l+1}(f)\coloneqq \lim_{n\to\infty}\frac{\lvert \exc^{\l+1}(f^n)\rvert}{n}
\]
the {\it dynamical number of the $(\l+1)$-exceptional locus}. This limit always exists, since the function $f\mapsto \lvert \exc^{\l+1}(f)\rvert$ is subadditive on $\Psaut^{\l}(f)$. The dynamical number of the $(\l+1)$-exceptional locus is invariant under conjugacy in $\Psaut^{\l}(X)$, since conjugation by a birational transformation $g$ changes the number of irreducible components of $\exc^{\l+1}(f)$ at most by a constant only depending on $g$ and $g^{-1}$.
 This number is inspired by the dynamical number of base points in the surface case. Nevertheless, it is a weaker notion in the sense that if $f\in\Bir(S )$, where $S $ is a regular projective surface over $\kk$, then $\nu^{1}(f)\leq \mu(f)$. For instance, the Hénon transformation $h\colon \p_\kk^2\dashrightarrow\p_\kk^2$ defined with respect to affine coordinates by $h(x,y)=(y,y^2+x)$ satisfies $\nu^{1}(f)=0$, whereas $\mu(f)=3$. 

In the same way as the blow-up complex gives a geometrical interpretation of the dynamical number of base points, the cube complex $\CC^{\l}(X )$ gives a geometrical interpretation of $\nu^{\l+1}(f)$ when $f\in \Psaut^{\l}(X )$.

Let us start with following lemma which is a direct consequence of \cite[Proposition 4, Corollary~5]{Liu-Sebag}. 

\begin{lemma}\label{lemme_composantes_irreductibles}
	Let $\kk$ be an algebraically closed field, let $X $ be a variety over $\kk$, and let $f\colon X \dashrightarrow X $ be a birational map that induces an isomorphism between dense open subvarieties $U $ and $V $ of $X $.
	The complements $X \setminus U $ and $X \setminus V $ have the same dimension, and the same number of irreducible components of maximal dimension.
\end{lemma}

\begin{proposition}\label{lemme_dynexcp}
	Let $X $ be a variety over a field $\kk$.
	For all $f\in\Psaut^\l(X )$ we have $\l(f)=\nu^{\l+1}(f)+\nu^{\l+1}(f^{-1})$. 
	Moreover, if $\kk$ is algebraically closed, then $\l(f)=2\nu^{\l+1}(f)$.
\end{proposition}

\begin{proof}By Corollary \ref{prop_semisimple_psaut}, for any $f\in\Psaut^\l(X )$ and any $x\in \Min(f)$, we have $n\in\N$, $\dist(x,f^n(x))=n\l(f)$.	
Let $p=(X, \id)$. Then, by Lemma \ref{distancel}, 
\[
\d(p, f^n(p))=\lvert \exc^{\l+1}(f^n)\rvert+\lvert \exc^{\l+1}(f^{-n})\rvert,
\] implying that for any $n\in \Z_{>0}$, $n\l(f)\leq \lvert \exc^{\l+1}(f^n)\rvert+\lvert \exc^{\l+1}(f^{-n})\rvert$. Taking the limit we get: $\l(f)\leq \nu^{\l+1}(f)+\nu^{\l+1}(f^{-1})$. 
	
On the other hand, let $x\in \Min(f)$ so for any $n\in\N$, $\dist(x,f^n(x))=n\l(f)$. Set $K:=\d(p, x)$. Then for any $n\in \N$, \[\lvert \exc^{\l+1}(f^n)\rvert+\lvert \exc^{\l+1}(f^{-n})\rvert=\d(p, f^n(p))\leq n\l(f)+2K\] and hence $\nu^{\l+1}(f)+\nu^{\l+1}(f^{-1})\leq \l(f)$. 

The second part of the statement follows from Lemma \ref{lemme_composantes_irreductibles}.
\end{proof}

When the field is not algebraically closed we do not know whether for all $f\in\Psaut^\l(X)$, $\Exc^{\l+1}(f)=\Exc^{\l+1}(f^{-1})$, even in the case where $\Exc^{\l+1}(f)$ is equal to zero.

Proposition \ref{lemme_dynexcp} implies the following proposition, which can be seen as an analogue to Theorem~\ref{blancdesthm}:

\begin{proposition}\label{prop_nu_0} Let $X $ be a variety over a field $\kk$.
	Let $f\in \Psaut^\l(X )$. Then 
	\begin{enumerate}
		\item $\nu^{\l+1}(f)$ is an integer;
		\item the sequence $\{\Exc^{\l+1}(f^n)\}_{n\in\N}$ is either bounded or grows asymptotically linearly in $n$;
	\item there exists a variety $Y $ and an isomorphism in codimension $\l$ $$\varphi\colon Y \dashrightarrow X $$ such that $|\Exc^{\l+1}(\varphi^{-1} f\varphi)|=\nu^{\l+1}(f)$ and $|\Exc^{\l+1}(\varphi^{-1} f^{-1}\varphi)|=\nu^{\l+1}(f^{-1})$;
		\item in particular, $\nu^{\l+1}(f)=0$ and $\nu^{\l+1}(f^{-1})=0$ if and only if there exists a variety $Y $ and an isomorphism in codimension $\l$ $$\varphi\colon Y \dashrightarrow X $$ such that $\varphi f\varphi^{-1}\in \Psaut^{\l+1}(Y)$.
	\end{enumerate}
\end{proposition}

\begin{proof}
	We choose a vertex $(Y ,\varphi)$ in $\Min(f)$. For any $n\in \N$, we have $n\lvert \exc^{\l+1}(f)\rvert\geq \lvert \exc^{\l+1}(f^n)\rvert$ and therefore:
	\begin{align*}
	\dist((Y ,\varphi),(Y ,f^n\varphi))&=n\left(\lvert \exc^{\l+1}(\varphi^{-1}f\varphi)\rvert+\lvert \exc^{\l+1}(\varphi^{-1}f\varphi)\rvert\right)\\
	&\geq\lvert \exc^{\l+1}(\varphi^{-1}f^n\varphi)\rvert+\lvert \exc^{\l+1}(\varphi^{-1}f^{-n}\varphi)\rvert\\
	& =\dist((Y ,\varphi),(Y ,f^n\varphi)).
	\end{align*} 
	It follows that $n\lvert \exc^{\l+1}(\varphi^{-1}f\varphi)\rvert=\lvert \exc^{\l+1}(\varphi^{-1}f^n\varphi)\rvert$. This implies that $\nu^{\l+1}(f)$, which is invariant under conjugation, is an integer. This proves the first three points. 
	
	The last point follows from Proposition~\ref{lemme_dynexcp} and point (3). 
\end{proof}

\subsection{(Pseudo)-regularization} 
In this section we gather results about regularization and pseudo-regularization. We also prove Theorem~\ref{thm_regularization}.

\begin{proposition}\label{boundedcomp}
	Let $X$ be a variety over a field $\kk$ and $G\subset\Psaut^\l(X )$ for some $0\leq \l\leq \dim(X )-1$. The subgroup $G$ is pseudo-regularizable in codimension $\l+1$ by an isomorphism in codimension $\l$ if and only if $\{\lvert\Exc^{\l+1}(g)\rvert\mid g\in G\}$ is uniformly bounded.
\end{proposition}

\begin{proof}
Consider the action of $G$ on $\CC^\l(X)$, then $G$ is pseudo-regularizable in codimension $\l+1$ by an isomorphism in codimension $\l$ if and only if $G$ fixes a vertex of $\CC^\l(X)$. By Proposition~\ref{fixedpoint}, $G$ is pseudo-regularizable in codimension $\l+1$ by an isomorphism in codimension $\l$ if and only if the orbits of $G$ are bounded, which is equivalent to $\{\lvert\Exc^{\l+1}(g)\rvert\mid g\in G\}$ being uniformly bounded, as a consequence of Lemma~\ref{distancel}. 
\end{proof}

%\begin{corollary}\label{prop_critere_pseudoreg}
%	Let $X $ be a normal variety over a field $\kk$ of dimension $d$ and let $G\subset\Psaut^\l(X )$. Then $G$ is pseudo-regularizable in codimension $\l+n\leq d$ if and only if there exists a sequence of varieties $X_{\l}=X $, $X_{\l+1}, \dots, X_{\l+n}$ and a sequence of isomorphisms in codimension $i$, $f_i\colon X_{i} \dashrightarrow X_{i+1}$ for $\l\leq i <\l+n$, such that the number of irreducible components of codimension $i+1$ of the exceptional locus of the elements of $f_if_{i-1}\dots f_\l G(f_if_{i-1}\dots f_\l)^{-1}$ is uniformly bounded for all $i$.
%\end{corollary}
%
%\begin{proof}The claim follows from applying Proposition~\ref{boundedcomp} inductively.
%\end{proof}

In the next proposition we obtain a similar result as Theorem \ref{prop_reg_field_extension}.

\begin{proposition}\label{fieldhigherdim}
Let $X $ be a geometrically integral variety of dimension $d$ over a perfect field $\kk$ and let $G\subset \Psaut^\l(X )$ be a subgroup. Then $G$ is pseudo-regularizable in codimension $\l+1$ over the algebraic closure $\bar{\kk}$ of $\kk$ by an isomorphism in codimension $\l$ if and only if $G$ is pseudo-regularizable in codimension $\l+1$ over $\kk$ by an isomorphism in codimension $\l$.
\end{proposition}

\begin{proof}
	If $G$ is pseudo-regularizable in codimension $\l+1$ over $\bar{\kk}$ by an isomorphism in codimension $\l$, then there exists a variety $Y_{\bar{\kk}}$ over $\bar{\kk}$ and an isomorphism in codimension $\l$ $f\colon X_{\bar{\kk}} \dashrightarrow Y_{\bar{\kk}}$ such that $fGf^{-1} \subset \Psaut^{\l+1}(Y_{\bar{\kk}} )$. In particular, by Proposition~\ref{boundedcomp} the number of irreducible components of pure codimension $\l+1$ of every element $g\in G$ considered as a birational transformation of $X_{\bar{\kk}}$ is uniformly bounded. The number of irreducible components of pure codimension $\l+1$ of $g$ considered as a birational transformation of $X_{\bar{\kk}}$ over $\bar{\kk}$ is larger or equal to the number of irreducible components of pure codimension $\l+1$ of $g$ considered as a birational transformation of $X $ over ${\kk}$. This implies that the number of irreducible components of pure codimension $\l+1$ of every element $g\in G$ considered as a birational transformation of $X $ is uniformly bounded, and hence, by Proposition~\ref{boundedcomp}, that $G$ is pseudo-regularizable in codimension $\l+1$ over~$\kk$ by a ${\kk}$-birational map that is an isomorphism in codimension $\l$.
	
	On the other hand, if $G$ is pseudo-regularizable in codimension $\l+1$ over $\kk$ by an isomorphism in codimension $\l$, then there exists a variety $Y$ and a birational transformation $f\colon X \to Y$ such that $fGf^{-1}\subset \Psaut^{\l+1}(Y)$. Since $\kk$ is perfect and $X$ is geometrically integral and birationally equivalent to $X$, we obtain that also $Y$ is geometrically integral (see \cite[p.~90, Remark~2.9, Example~2.10]{Liu}), and therefore, that $Y_{\bar{\kk}}$ is a variety over $\bar{\kk}$. In particular, we can consider $fGf^{-1}$ as a subgroup of $\Psaut^{\l+1}(Y_{\bar{\kk}})$, so $G$ is pseudo-regularizable in codimension $\l+1$ over $\bar{\kk}$ by a $\bar{k}$-birational map that is an isomorphism in codimension $\l$ . 
\end{proof}

In what follows, an algebraic group over the field $\kk$ is always assumed to be of finite type, but it doesn't need to be irreducible nor smooth.

\begin{lemma}\label{rational_boundedcomp}
	Let $X $ be a variety over a field $\kk$ and let $G$ be an algebraic group with a faithful rational action on $X $, so we can consider $G(\kk)$ to be a subgroup of $\Bir(X )$. Then there exists a constant $C$, such that for every $g\in G(\kk)$ the number of irreducible components of the exceptional locus $g$ of codimension $1$ is bounded by~$C$.
\end{lemma}

\begin{proof}
	By definition, the rational action is given by a $G$-birational map $\varphi\colon G\times X \dashrightarrow G\times X $ such that $\{g\}\times X $ is not contained in the exceptional locus of $\varphi$ for all $g\in G$. For each $g\in G$, if $(g,x)$ is not contained in the exceptional locus of $\varphi$, then $g$ induces a local isomorphism at $x\in X $. In other words, if $\pi\colon G\times X\to X $ is the second projection, then $\Exc(g)\subset \pi(\{g\}\times X\cap \Exc(\varphi))$ for all $g\in G$. The number of irreducible components of codimension $1$ of $\pi(\{g\}\times X \cap \Exc(\varphi))$ is bounded independently of $g$.
\end{proof}

As a special case we obtain the following weak version of Weil's regularization theorem:
\begin{proposition}
	Let $X $ be a variety over a field $\kk$ and $G$ an algebraic group with a faithful rational action on $X $ then $G(\kk)$ is pseudo-regularizable in codimension~$1$.
\end{proposition}

\begin{proof}
	The number of irreducible components of codimension $1$ of the exceptional locus of every element $g\in G$ is bounded, by Lemma~\ref{rational_boundedcomp}. Proposition \ref{boundedcomp} implies that $G(\kk)$ is pseudo-regularizable.
\end{proof}

Finally, let us also prove Theorem~\ref{thm_regularization}.

\begin{proof}[Proof of Theorem \ref{thm_regularization}]
	The group $G$ acts on $\CC^{0}(X )$ by isometries and fixes therefore, by definition of the property FW, a vertex $[(X_{1},\varphi_1)]$, hence $\varphi_1$ conjugates $G$ to a subgroup of $\Psaut^{1}(X_{1})$. This yields an action of $G$ on $\CC^{1}(X_{1})$ and again, since $G$ has the property FW, it fixes a vertex in $\CC^{1}(X_{1})$ and the image of $G$ in $\Psaut^{1}(X_{1})$ is therefore conjugate by a isomorphism in codimension one $\varphi_2$ to a subgroup of $\Psaut^{2}(X_{2})$ for some variety $X_{2}$. We continue inductively until we find a variety $X_{d}$ and a birational transformation $\varphi=\varphi_1\dots\varphi_{d}\colon X_{d}\dashrightarrow X $ that conjugates $G$ to a subgroup of $\Psaut^{d}(X_{d})=\aut(X_{d})$.
	If $X$ is a  surface, Proposition~\ref{lemme_FW_surface} implies that $G$ is moreover projectively regularizable.

	If $g$ is divisible or distorded, consider the action of $\langle g\rangle$ on $\CC^0(X )$. We define the length function $L(f):=\dist((X,\id),(X,f))$ on $\Bir(X)$. Let us observe that $L$ is subadditive, takes integral values, and the map $n\mapsto L(g^n)$ grows asymptotically like $l(g)n$, where $l(g)$ is the translation length of $g$. Since $g$ is, by assumption, distorted or divisible, the restriction of $L$ to $\langle g\rangle$ has to be bounded. This implies that the orbit of $\langle g\rangle$ on the $\CAT(0)$ cube complex is bounded, hence it fixes a vertex $[(X_{1},\varphi_1)]$ by Proposition \ref{fixedpoint}. This yields an action of $G$ on $\CC^{1}(X_{1})$ and again, since $g$ is divisible or distorted, it fixes a vertex in $\CC^{1}(X_{1})$ and the image of $\langle g\rangle$ in $\Psaut^{1}(X_{1})$ is therefore conjugate by a isomorphism in codimension one $\varphi_2$ to a subgroup of $\Psaut^{2}(X_{2})$ for some variety $X_{2}$. We continue inductively until we find a variety $X_{d}$ and a birational transformation $\varphi=\varphi_1\dots\varphi_{d}\colon X_{d}\dashrightarrow X $ that conjugates $\langle g\rangle$ to a subgroup of $\Psaut^{d}(X_{d})=\aut(X_{d})$. Hence, it is regularizable.
	If $X$ is a surface, using $\Cb(X)$ we obtain similarly that $g$ is projectively regularizable. 
\end{proof}

\subsection{Degree growth and centralizers of loxodromic elements}
In this section, we prove Theorem \ref{centralizer} and Theorem \ref{thm_degree_growth}. 

\begin{proof}[Proof of Theorem \ref{centralizer}]
	Since $f$ is not pseudo-regularizable, by Proposition \ref{prop_semisimple_psaut} it acts as a loxodromic isometry on $\CC^0(X)$. We may assume that the vertex $(X,\id)$ is contained in  an axis of $f$. This implies in particular the following for all $n>0$:
	\begin{itemize}
		\item $\Exc^1(f^n)$ and $\Exc^1(f^{-n})$ are disjoint,
		\item $\Exc^1(f^n)=\Exc^1(f^{n-1})\sqcup f^{-(n-1)}(\Exc^1(f))$.
	\end{itemize}
	 Assume that $\Exc^1(f)$ consists of $r$ irreducible hypersurfaces $E_1,\dots, E_r$. Then for every irreducible hypersurface $H $ in $\Exc^1(f^n)$ there exists a $1\leq i\leq r$ and a $1\leq j\leq n-1$ such that $H=f^{-j}(E_i)$, i.e., every irreducible hypersurface $H$ is contained in a forward orbit of one of the $E_i$ under $f^{-1}$.
	
	Let $g\in \Bir(X)$ be an element of infinite order that centralizes $f$. Let $n>0$ be large enough such that for all $m\geq 0$ the two sets $\Exc^1(f^{n+m})\setminus \Exc^1(f^n)$ and $\Exc^1(f^{-n-m})\setminus \Exc^1(f^{-n})$ do not contain any hypersurface contained in $\Exc^1(g^i)$ for all $-r-1\leq i\leq r+1$. For $m$ large, $g^i$ maps most of the irreducible hypersurfaces in the exceptional locus of $f^{n+m}$ to irreducible hypersurfaces. Since $g$ commutes with $f$, we have $\Exc^1(f^{l}g^i)=\Exc^1(g^if^{l})$ for all $l$ and all $-r-1\leq i\leq r+1$. We know that $\Exc^1(f^{l}g^i)\subset g^{-i}(\Exc^1(f^l))\cup\Exc^1(g^i)$ and since the hypersurfaces $\Exc^1(f^{n+m})\setminus \Exc^1(f^n)$ lie outside the exceptional locus of the $g^{-i}$, we obtain that for all but $K$ hypersurfaces, where $K$ does not depend on $m$, $g^{-i}(\Exc^1(f^{n+m})\setminus \Exc^1(f^n))$ is contained in $\Exc^1(f^{n+m})$. If $m$ is large enough there has therefore to exist an irreducible hypersurface $H$ contained in $\Exc^1(f^{n+m})\setminus \Exc^1(f^n)$ that is mapped by all the $g^i$ to other irreducible hypersurfaces in $\Exc^1(f^{n+m})\setminus \Exc^1(f^n)$. More precisely, for each $1\leq i\leq r+1$ there is an irreducible hypersurface $H_i'$ contained in $\Exc^1(f^{n+m})\setminus \Exc^1(f^n)$ such that $g^i(H)=H_i'$. As observed above, each of the hypersurfaces $H_i'$ is contained in a forward orbit under $f^{-1}$ of the irreducible hypersurfaces $E_1,\dots, E_r$. By the pigeonhole principle, there exist $i\neq j$, $1\leq i,j\leq r+1$ such that $H_i'$ and $H_j'$ are contained in the forward orbit of the same hyperplane $E_l$. In particular, there exists a $s<0$ such that $f^s(H_i')=H_j'$ or $f^s(H_j')=H_i'$. Without loss of generality we may assume $f^s(H_i')=H_j'$. In other words, $f^s(g^i(H))=g^j(H)$. Hence, for the birational transformation $h=f^sg^{i-j}$ we have $h(g^j(H))=g^j(H)$. Note that $h$ still centralizes $f$. So for all $m$ large enough we have $h(f^{-m}(g^j(H)))=f^{-m}(hg^j(H))=f^{-m}(g^j(H))$. Note that $f^{-m}(g^j(H))=g^j(f^{-m}(H))$. Since $f^{-m}(H)\neq f^{-m'}(H)$ for all $m\neq m'$, we therefore obtain that the hypersurfaces $f^{-m}(g^j(H))$ are all different for $m$ large enough. In other words, there exist infinitely many hypersurfaces $H'$ such that $h(H')=H'$. By \cite[Corollary~1.3]{bell2018invariant} (see also \cite{Cantat_invariant} for the case, where the base field is the field of complex numbers), there exists a non-constant rational function in the function field $\gamma\in\kk(X)\setminus\kk$, such that $\gamma h=\gamma$. Consider the function field $\kk(X)^h\subset\kk(X)$ of functions on $X$ that are invariant under $h$. Since $f$ commutes with $h$, the subfield $\kk(X)^h$ is preserved by $f$. The equivariant embedding $\kk(X)^h\hookrightarrow\kk(X)$ corresponds to a rational dominant map $\varphi\colon X\dashrightarrow Y$ to some variety $Y$ that is $h$-invariant and the action of $f$ on $\kk(X)^h$ induces a birational map on $Y$ such that $\varphi$ is $f$-equivariant. By the existence of $\gamma$, the variety $Y$ is positive dimensional. If $h$ has infinite order, we have moreover $\dim(Y)<\dim(X)$, since in this case, $\kk(X)$ is an infinite extension of $\kk(Y)$. Otherwise, there exists an integer $m>0$ such that $g^m$ is contained in the cyclic group generated by $f$. So if $f$ does not preserve any rational fibration, then there exists for every element $g\in\cent(f)$ an integer $m>0$ satisfying $g^m\in\langle f\rangle$. By \cite[Theorem~6.4]{genevois2019cubical}, $\langle f\rangle$ is a direct factor of some finite index subgroup of $\cent(f)$, which implies in particular that $\cent(f)$ contains as a finite index subgroup $\langle f\rangle\times G$, where $G\subset\Bir(X)$ is a torsion subgroup. 
\end{proof}

Note that Theorem~\ref{centralizer} implies in particular that if a birational transformation $g\in\Bir(X)$ inducing a loxodromic isometry on $\CC^0(X)$ is irreducible in the sense that it does not permute the fibers of any rational map $X\dashrightarrow Y$, where $\dim(Y)<\dim(X)$, and $f$ is a transformation that commutes with $g$, then there exist non-zero integers $m,n$ such that $f^m=g^n$. A similar result has been proven by Serge Cantat for $f\in\Cr_2(\mathbb{C})$ that induce a loxodromic isometry in $\h$. In fact, the proof of Theorem~\ref{centralizer} was inspired by Serge Cantat's proof of Proposition~8.8. in the long version of the paper \cite{Cantat_groupes_birat}, which can be found on the author's personal website.

\begin{proof}[Proof of Theorem \ref{thm_degree_growth}]
	We may replace $\kk$ by its algebraic closure. This doesn't change the degrees of $g$ and its iterates and the induced isometry of $g$ on $\CC^0(\p_\kk^d)$ is still loxodromic. Assume that $g$ is of degree $n$, i.e., $g$ is given by $[x_0:\dots:x_d]\dashmapsto [g_0:\dots: g_d]$, where the $g_i\in\kk[x_0,\dots, x_d]_n$ are homogeneous polynomials of degree $n$ without a non-constant common factor. Then the exceptional locus of $g$ is the zero locus of the Jacobian of $(g_0,\dots, g_d)$, which is of degree $\leq (d+1)(n-1)$. In particular, if the exceptional locus of $g$ has $r$ irreducible components, then $n$ is bounded from below by $\frac{1}{d+1} r$. By assumption, $g$ induces a loxodromic isometry on $\CC^0(\p_\kk^d)$. Since we work over the algebraic closure of $\kk$, Lemma~\ref{lemme_composantes_irreductibles} implies that the number of irreducible components of $\Exc(g)$ is the same as the number of irreducible components of $\Exc(g^{-1})$. By Lemma~\ref{distancel}, the number of irreducible components of the exceptional locus of $g^n$ grows therefore asymptotically like $C\cdot n$ for some integer $C>0$. As a consequence, the degree of $g^n$ grows asymptotically at least like $\frac{C}{d+1} n$. 
\end{proof}

Another application of the action on a $\CAT(0)$ cube complex is the following direct corollary to a result by Anthony Genevois.

\begin{proposition}
	Let $X$ be a variety over a field $\kk$. 
	Let $G\subset \Bir(X)$ be a polycyclic group that doesn't contain any non-trivial pseudo-regularizable element. Then $G$ is virtually free abelian.
\end{proposition}

\begin{proof}
	If $G$ does not contain any pseudo-regularizable element, then its action on $\CC^0(X)$ is proper. The result follows directly from \cite[Corollary~6.9]{genevois2019cubical} (see also \cite{cornulier_wallings}). 
\end{proof}

\section{Further remarks and open questions}\label{section_further_remarks}

\subsection{Comparison of isometry types in the surface case}\label{comparison}
One of the breakthroughs in the understanding of the structure of groups of birational transformations in rank two has been their action on an infinite dimensional hyperbolic space $\h$. Let $S$ be a regular surface over a field $\kk$ and $H$ an ample divisor on $S$, then the \emph{degree} of $f$ with respect to $H$ is defined by $\deg_H(f)=f^*H\cdot H$, where $f^*H$ denotes the total transform. We recall the following central result. 

\begin{theorem}[{\cite{MR563788}, \cite{diller2001dynamics},\cite{Cantat_groupes_birat}}]\label{deggrowth}{Let $S$ be a regular surface over a field $\kk$ and $H$ an ample divisor on $S$. Consider the action of $\Bir(S)$ on $\h$. For this action, 
		\begin{enumerate}
			\item elliptic isometries correspond to birational transformations for which the sequence $(\deg_H(f^n))$ is bounded;
			\item parabolic isometries correspond to a polynomial growth of $(\deg_H(f^n))$: either $\deg_H(f^n)$ grows asymptotically linearly and $f$ is
			called a Jonqui\`eres twist, or $\deg_H(f^n)$ grows asymptotically quadratically and $f$ is called a Halphen twist;
			\item loxodromic isometries correspond to birational transformations for which $\deg_H(f^n)$ grows asymptotically exponentially. 
	\end{enumerate}}
\end{theorem}

In what follows, we look at the various types of birational transformations as classified by Theorem~\ref{deggrowth} and compare the types of the induced actions on the three non-positively curved spaces $\h$, $\Cb(S)$, and $\CC^0(S)$.

\subsubsection*{Transformations with bounded degree growth} Transformations with bounded degree growth, sometimes also called \emph{algebraic} transformations, induce elliptic isometries on $\h$, by Theorem~\ref{deggrowth} and are conjugate to the automorphism of a projective surface. Hence, they also induce elliptic isometries on $\Cb(S)$ and $\CC^0(S)$. 

\subsubsection*{Transformations with linear degree growth} These transformations are exactly the de Jonqui\`eres twists and are never conjugate to an automorphism of any projective surface (\cite{Blanc_Deserti}). So they induce loxodromic isometries on $\Cb(S)$. There exist de Jonqui\`eres twists inducing elliptic isometries on $\CC^0(S)$, such as the transformation of $\A^2$ given $(x,y)\dashmapsto(xy,y)$ and other de Jonqui\`eres twists inducing loxodromic isometries on $\CC^0(S)$ such as the birational transformation of $\A^2$ given by $(x,y)\mapsto (xy, y+1)$. 

\subsubsection*{Transformations with quadratic degree growth} These transformations are exactly the Halphen twists, they induce parabolic isometries on $\h$ and are conjugate to automorphisms of projective surfaces, by Theorem~\ref{deggrowth}. Hence they induce elliptic isometries on $\Cb(S)$ and $\CC^0(S)$. 

\subsubsection*{Transformations with exponential degree growth} This is the generic type of birational transformations of rational surfaces. They induce loxodromic isometries on $\h$. Some of them are conjugate to an automorphism of a projective surface, while others are not. The {\it dynamical degree}, which is defined as the limit $\lambda(f)=\lim_{n\to\infty}\deg(f^n)^{1/n}$, has been studied intensely and is connected to the question whether the transformation $f$ is conjugate to an automorphism of a projective surface (\cite{blanc_cantat_dynamic_degree}). The transformations with exponential degree growth that are conjugate to an automorphism of a projective surface, induce elliptic isometries on $\Cb(S)$ and $\CC^0(S)$. The transformations with exponential degree growth that are not conjugate to an automorphism of a projective surface induce loxodromic isometries on $\Cb(S)$, whereas they induce elliptic isometries on $\CC^0(S)$ if they are regularizable, and loxodromic isometries on $\CC^0(S)$ if they aren't. For instance, the transformation $(x,y)\dashmapsto (y, x+y^2)$ induces an elliptic isometry on $\CC^0(S)$, whereas the transformation $(x,y)\dashmapsto(x^2y, x(y+1))$ induces a loxodromic isometry on $\CC^0(S)$. 

The table above summarizes the situation and lists also the three dynamical invariants $\lambda(f), \mu(f),$ and $\nu^1(f)$.

\begin{table}[]
	\resizebox{\textwidth}{!}{%
		\begin{tabular}{@{}lllllll@{}}
			\toprule
			degree growth & $\lambda(f)$ & $\mu(f)$ & $\nu^1(f)$ & Isometry on $\h$ & Isometry on $\Cb(S)$ & Isometry on $\CC^0(S)$ \\ 	\toprule 
			bounded    & 1      & 0    & 0    & elliptic       & elliptic           & elliptic            \\ 	\midrule 
			linear    & 1      & $>0$   & 0    & parabolic       & loxodromic          & elliptic            \\ 	\midrule
			linear    & 1      & $>0$   & $>0$   & parabolic       & loxodromic          & loxodromic           \\ 	\midrule
			quadratic   & 1      & 0    & 0    & parabolic       & elliptic           & elliptic            \\ 	\midrule 
			exponential  & $>1$     & 0    & 0    & loxodromic      & elliptic           & elliptic            \\ 	\midrule 
			exponential  & $>1$     & $>0$   & 0    & loxodromic      & loxodromic          & elliptic            \\ 	\midrule 
			exponential  & $>1$     & $>0$   & $>0$   & loxodromic      & loxodromic          & loxodromic           \\ 	\midrule
		\end{tabular}%
	}
	\vspace{3mm}
	\caption{The type of the induced isometries of a birational transformation $f$ of a complete surface $S$.}
	\label{table}
\end{table}

\subsection{Variations on the cube complexes} 
In this paper we set our focus on varieties. However, it seems that similar cube complexes can be constructed in the analytic category such that results analogous to the results in this paper for varieties can be proved for compact complex manifolds or analytic spaces. 

There also exist some variations in the construction of our cube complexes. For instance, if $X$ is a regular variety, then we can construct the complexes $\CC^\l(X)$ in such a way that all the vertices are marked regular varieties, or marked normal varieties. 
\subsection{Groups of elliptic elements} Unfortunately, we were not able to answer the following intriguing question:

\begin{question}\label{ques:elliptic}
	Let $S$ be a regular projective surface and $\Gamma\subset\Bir(S)$ a finitely generated subgroup such that every element in $\Gamma$ is projectively regularizable. Does this imply that $\Gamma$ is projectively regularizable?
\end{question} 

With regards to the action of $\Bir(S)$ on $\Cb(S)$, Question~\ref{ques:elliptic} translates to the question whether a finitely generated subgroup $\Gamma\subset\Bir(S)$ containing only elliptic elements, has a fixed point. A finitely generated group acting on a finite dimensional $\CAT(0)$ cube complex by elliptic isometries always has a fixed point (\cite{Sageev-ends_of_groups}), whereas there exist examples of finitely generated subgroups acting on infinite dimensional $\CAT(0)$ cube complexes by elliptic isometries without having a fixed point (see for example \cite{Osajda_2018}).

One can show using results about the structure of automorphism groups of projective varieties, that a positive answer to Question~\ref{ques:elliptic} would imply that a finitely generated subgroup of $\Bir(S)$ such that all elements are \emph{algebraic}, i.e., their degree is bounded under iteration, is of bounded degree itself. This would answer a question by Serge Cantat and Charles Favre (\cite{Cantat_groupes_birat}, \cite{favrebourbaki}) positively.

The condition in the conjecture that $\Gamma$ is finitely generated is necessary. For instance, the group $G\coloneqq\{(x, y+f(y))\mid f\in\kk(y)\}$ consists of elements inducing elliptic isometries on $\Cb(\p_\kk^2)$, but one can show that $G$ is not projectively regularizable. 

The following observation treats a very specific case of Question~\ref{ques:elliptic}:

\begin{proposition}
	Let $S$ be a regular projective surface and let $\Gamma\subset\Bir(S)$ be a subgroup generated by two elements of order $2$ such that the composition of the two generators is projectively regularizable. Then $\Gamma$ is projectively regularizable.
\end{proposition}
\begin{proof}
	Let $f,g$ be generators of $\Gamma$ of order two. Assume that $fg$ is projectively regularizable.
	By assumption, any element of $\Gamma$ can be written, for some $i\geq 0$ and $j\in\{0,1\}$, either as $g^j(fg)^i$ or $(fg)^if^j$. Consider the action of $\Gamma$ on $\Cb(S)$ and let $(T,\varphi)$ be a vertex that is fixed by $fg$. Then the orbit of the vertex $(T,\varphi)$ by $\Gamma$ is bounded by $\max\left(\dist\left((T,\varphi),(T,g\varphi)\right),\dist\left((T,\varphi),(T,f\varphi)\right)\right)$. Hence, the action of the group $\Gamma$ has a fixed point by Proposition \ref{fixedpoint} and so it is projectively regularizable.
\end{proof}

\subsection{Geometry of $\CC^1(X)$}

In Remark~\ref{rmq_plusieurs_cubes} we observed that the cube complexes $\CC^1(X)$ have a more interesting geometry than the complexes $\CC^0(X)$, where every finite set of vertices is contained in a single cube. It could be interesting to better understand the geometry of $\CC^1(X)$ in order to obtain new tools to study groups of pseudo-automorphisms. One can also dream about using $\CC^1(X)$ to obtain insights into sequences of flips and flops.

\bibliographystyle{amsalpha}
\bibliography{bibliography_cu}
\end{document}